\RequirePackage{fix-cm}
\documentclass[smallextended]{svjour3}       % onecolumn (second format)
\smartqed  % flush right qed marks, e.g. at end of proof
\usepackage{graphicx}
\usepackage{pdfsync}
\usepackage[hidelinks]{hyperref}
\usepackage{latexsym,amsfonts,amsmath,graphics,multirow}
\usepackage{mathrsfs}

\usepackage[normalem]{ulem}
\usepackage[usenames,dvipsnames,svgnames,table]{xcolor}

\usepackage{epsfig,tikz,pgfplots}
\usetikzlibrary{positioning}
\usepackage{enumitem,amssymb}
\usepackage{algorithm,algorithmic}
\usepackage{float}
\usepackage{makecell}
\usepackage[font=scriptsize]{caption}

\spnewtheorem{theorem}{Theorem}[section]{\bfseries}{\itshape}
\spnewtheorem{lemma}[theorem]{Lemma}{\bfseries}{\itshape}
\spnewtheorem{corollary}[theorem]{Corollary}{\bfseries}{\itshape}
\spnewtheorem{definition}[theorem]{Definition}{\bfseries}{\itshape}
\spnewtheorem{assumption}[theorem]{Assumption}{\bfseries}{\itshape}
\spnewtheorem{example}[theorem]{Example}{\itshape}{\itshape}
\spnewtheorem{remark}[theorem]{Remark}{\itshape}{\itshape}
\numberwithin{equation}{section}
\numberwithin{theorem}{section}
\newcommand{\mask}[1]{}

\definecolor{darkred}{RGB}{139,0,0}
\definecolor{darkgreen}{RGB}{0,100,0}
\definecolor{darkmagenta}{RGB}{139,0,139}
\definecolor{orange}{RGB}{207,83,0}
\definecolor{brown}{RGB}{139,69,19}

\newcommand{\vk}[1]{{\color{black}{#1}}}

\newcommand{\bstheta}{{\boldsymbol{\theta}}}

\newcommand{\bsgamma}{{\boldsymbol{\gamma}}}
\newcommand{\bsalpha}{{\boldsymbol{\alpha}}}
\newcommand{\bseta}{{\boldsymbol{\eta}}}
\newcommand{\bsnu}{{\boldsymbol{\nu}}}

\newcommand{\bsb}{{\boldsymbol{b}}}

\newcommand{\bsm}{{\boldsymbol{m}}}

\newcommand{\bst}{{\boldsymbol{t}}}
\newcommand{\bsx}{{\boldsymbol{x}}}
\newcommand{\bsy}{{\boldsymbol{y}}}
\newcommand{\bsz}{{\boldsymbol{z}}}

\newcommand{\bslambda}{{\boldsymbol{\lambda}}}

\newcommand{\setu}{\mathrm{\mathfrak{u}}}
\newcommand{\setv}{\mathrm{\mathfrak{v}}}

\newcommand{\bsxi}{{\boldsymbol{\xi}}}

\definecolor{cs}{rgb}{0,0,0}
\newcommand{\cs}[1]{{\color{cs}{#1}}}
\usepackage{subfig,microtype}
\journalname{Numerische Mathematik}
\newcommand{\rev}[1]{{{#1}}}
\allowdisplaybreaks
\begin{document}

%\thanks{Grants or other notes
\title{Quasi-Monte Carlo for Bayesian design of experiment problems governed by parametric PDEs%\thanks{Grants or other notes
%about the article that should go on the front page should be
%placed here. General acknowledgments should be placed at the end of the article.}
}
%\subtitle{Do you have a subtitle?\\ If so, write it here}

\titlerunning{Quasi-Monte Carlo for Bayesian design of experiment problems}
% if too long for running head

\author{
Vesa Kaarnioja\and Claudia Schillings}

\authorrunning{V. Kaarnioja and C. Schillings}
 % if too long for running head

\institute{
              \rev{Vesa Kaarnioja}\at
              \rev{School of Engineering Sciences, LUT University, P.O.~Box 20, 53851 Lappeenranta, Finland\\
			  E-mail:
             {\tt vesa.kaarnioja@lut.fi}}\\
             Claudia Schillings\at
              Fachbereich Mathematik und Informatik, Freie Universit\"at Berlin, Arnimallee~6, 14195 Berlin, Germany\\
			  E-mail:
             {\tt c.schillings@fu-berlin.de}
                        %\\
               %  \\
%             \emph{Present address:} of F. Author  %  if needed
}

\date{Received: date / Accepted: date}
% The correct dates will be entered by the editor

\maketitle

\begin{abstract}
\begin{sloppypar}
\vk{This paper contributes to the study of optimal experimental design for Bayesian inverse problems governed by partial differential equations (PDEs). We derive estimates for the parametric regularity of multivariate double integration problems over high-dimensional parameter and data domains arising in Bayesian optimal design problems. We provide a detailed analysis for these double integration problems using two approaches: a full tensor product and a sparse tensor product combination of quasi-Monte Carlo (QMC) cubature rules over the parameter and data domains. Specifically, we show that the latter approach significantly improves the convergence rate, exhibiting performance comparable to that of QMC integration of a single high-dimensional integral. Furthermore, we numerically verify the predicted convergence rates for an elliptic PDE problem with an unknown diffusion coefficient in two spatial dimensions, offering empirical evidence supporting the theoretical results and highlighting practical applicability.}
\end{sloppypar}
% \PACS{PACS code1 \and PACS code2 \and more}
\end{abstract}

\section{Introduction}
Optimal experimental design involves designing a measurement configuration, e.g., optimal placement of sensors to collect observational data, which maximizes the information gained from the experiments~\cite{atkinson2007optimum,watson1987foundations,Ucinski}. By carefully designing the experiments, optimal experimental design aims to enhance the precision and efficiency of the data collection process. This methodology plays a pivotal role in various fields,  e.g., in engineering, but also social sciences and environmental studies. 

We will focus on optimal experimental design 
for Bayesian inverse problems governed by partial differential equations (PDEs) with high or infinite-dimensional parameters~\cite{alexand,10.1214/ss/1177009939}. The Bayesian approach incorporates prior knowledge and beliefs into the design process. Bayesian optimal design aims to maximize the information gained from the data while minimizing resources and costs. This approach is particularly useful in situations where the sample size is limited or when there are complex relationships between parameters. We consider as a criterion for the information gain the Kullback--Leibler divergence between the prior and posterior distribution (the solution of the underlying Bayesian inverse problem) and maximize the expected information gain, i.e., the average information gain with respect to all possible data realizations.

From a computational point of view, the Bayesian optimal design is challenging as it involves the computation or approximation of the  expected utility, in our case the expected information gain. These challenges are primarily rooted in the high dimensionality of the parameters involved in the inversion process, the substantial computational cost associated with simulating the underlying model, and the inaccessibility of the joint parameters and data distribution. By exploiting the problem structure of the forward problem, we will address these computational challenges and propose a quasi-Monte Carlo (QMC) method suitable for the infinite-dimensional setting.

\subsection{Literature overview}

The Bayesian approach to inverse problems governed by PDEs has become very popular over the last years. Mathematical modeling of physical phenomena described by PDEs often involves a high or even infinite number of parameters, which need to be estimated from the data. The Bayesian framework provides a systematic approach for quantifying uncertainties and updating model parameters using observed data. We refer to~\cite{stuart_2010} for an overview on the mathematical foundation and computational methods in this context. As the data collection process is often very expensive, one is naturally interested in optimizing this process, i.e., finding a setup such that the information about the unknown parameters is maximized. Optimal experimental design is a crucial concept in the field of statistics and applied mathematics involving strategically planning experiments to extract the maximum amount of information with the fewest resources, see, e.g.,~\cite{atkinson2007optimum,Ucinski,watson1987foundations} for a general introduction and overview and~\cite{alexand,10.1214/ss/1177009939} for the Bayesian approach to optimal design. The development of fast computational algorithms for the solution of the Bayesian optimal design problem for models described by PDEs is crucial to ensure the feasibility for applications. However, the tools available for Bayesian optimal experimental design  in the case of models governed by PDEs  are typically limited to specific scenarios and often lack comprehensive convergence analysis, cf.~\cite{alexand}. Huan et al. \cite{huan2013simulation} present an alternative simulation-based approach tailored for optimal Bayesian experimental design within the realm of nonlinear systems. Their methodology employs a double-loop Monte Carlo technique, polynomial chaos approximation of the parameter-to-observation map, and simultaneous stochastic approximation.

In the context of linear problems, some progress has been made from both theoretical and numerical perspectives. Alexanderian et al.~\cite{AlexanderianEtAl2014} address A-optimal design of experiments for infinite-dimensional Bayesian linear inverse problems. Their work incorporates techniques such as low-rank approximation of the parameter-to-observable map and a randomized trace estimator for efficient objective function evaluation. Achieving sparsity in sensor configuration is facilitated through the utilization of penalty functions. Existing methods often rely on Laplace approximations of distributions.

Beck et al.~\cite{BeckEtAl2018} propose an efficient Bayesian experimental design approach that utilizes Laplace-based importance sampling to compute the expected information gain. They explore the effectiveness of the double-loop Monte Carlo method, with a specific focus on Laplace-based techniques. Despite the convergence of Laplace approximation to the posterior under suitable assumptions, the convergence analysis of Laplace-based Bayesian optimal experimental design, i.e., its incorporation as an approximation of the posterior rather than a preconditioner, necessitates non-asymptotic bounds that are currently unavailable.
In \cite{DBLP:conf/icml/RainforthCYW18}, a nested Monte Carlo strategy has been suggested for Bayesian experimental design, which, under regularity assumptions on the forward problem, can recover the original Monte Carlo rate. For QMC, a similar approach has been proposed in \cite{bartuska2024doubleloop} achieving rates up to $2/3$ in terms of the number of forward function evaluations with constants depending on the parameter and observation space dimension.

In a large-scale Bayesian optimal experimental design approach~\cite{WuEtAl2022}, a derivative-informed projected neural network is employed. The parameter-to-observation map is approximated using neural networks. While numerical experiments demonstrate the efficiency of the method for specific test cases, a convergence analysis of the proposed method is currently lacking. In \cite{koval2024tractable}, a transport-map-based surrogate to the joint probability law is proposed, where the complexity is reduced by using tensor trains. \cs{In the context of optimization under uncertainty, which is very closely related to the optimal design problem, a one-shot framework can be shown to significantly reduce the computational costs \cite{Guth2021}.}

We will focus here on QMC methods for the approximation of the integrals. When dealing with integrands that are sufficiently smooth, it becomes possible to formulate QMC rules with error bounds independent of the number of stochastic variables, achieving faster convergence rates compared to Monte Carlo methods. Consequently, QMC methods have demonstrated considerable success in applications involving PDEs with random coefficients, as evidenced in works such as~\cs{\cite{herrmann2,GGKSS2019,gilbert,harbrecht,herrmann3,kuonuyenssurvey,KSSSU2015,KSS2012}}. They have proven especially effective in the realm of PDE-constrained optimization under uncertainty, as highlighted in~\cite{GKKSS2019,guth2022parabolic,Andreas3}.

%\vspace*{-.5cm}

\subsection{Outline of the paper}
In this paper, we analyze the Bayesian optimal design problem for model problems satisfying certain parametric regularity bounds. In particular, we make the following contributions:
\begin{itemize}
    \item We establish parametric regularity of the integrand for the Bayesian optimal design problem. To be more precise, the analysis is presented for the integrands with respect to parameters and data.
    \item We present an error analysis for the full tensor QMC method. We prove that the regularity of the forward problem leads to dimension-independent convergence rates (with respect to the parameters). In addition, we discuss and analyze a sparse tensor approach, which allows to improve the rate significantly while preserving the dimension-robustness. We show that the performance is comparable to that of a single integral. Note that the proposed approach is also applicable for other sampling strategies and therefore the analysis is on its own interesting for Bayesian optimal design.
    \item We numerically verify the predicted convergence rates for an elliptic PDE problem subject to an unknown diffusion coefficient.%
\end{itemize}

The remainder of this article is structured as follows. We first introduce the notation in the following subsection.
The problem setting for Bayesian optimal design is presented in Section~\ref{sec:problem} while %
Section~\ref{sec:qmc} gives an overview of QMC integration. The model problem and the corresponding optimal design problem is discussed in Section~\ref{sec:model}. We then present the regularity analysis for our model problem in Section~\ref{sec:reg}. This forms the basis for the error analysis in Section~\ref{sec:single}. Sections~\ref{sec:ftp} and~\ref{sec:stp} contain the main results for the full tensor and sparse tensor cubature. We illustrate the theoretical results with numerical experiments presented in Section~\ref{sec:numex}. We summarize the main results and also give an outlook to future work in Section~\ref{sec:conclusions}. The Appendices~\ref{sec:hypergeometric} and~\ref{sec:periodic} contain technical results needed for the regularity analysis and a summary of our main parametric regularity results for a periodic transformation of the model problem. 

\subsection{Notations and preliminaries}
Let $\bsnu=(\nu_j)_{j=1}^s\in\mathbb N_0^s$, $\bsm=(m_j)_{j=1}^s\in\mathbb N_0^s$, and let $\bsx=(x_j)_{j=1}^s$ be a sequence of real numbers. We define the notations
\begin{align*}
&\bsnu\leq \bsm\quad\text{if and only if}\quad \nu_j\leq m_j~\text{for all}~j\in\{1,\ldots,s\},\\
&|\bsnu|:=\sum_{j=1}^s\nu_j,\quad \partial_{\bsx}^{\bsnu}:=\prod_{j=1}^s \frac{\partial^{\nu_j}}{\partial x_j^{\nu_j}},\quad \binom{\bsnu}{\bsm}:=\prod_{j=1}^s\binom{\nu_j}{m_j},\\
&\bsnu!:=\prod_{j=1}^s\nu_j!,\quad \boldsymbol x^{\bsnu}:=\prod_{j=1}^sx_j^{\nu_j},
\end{align*}
where we use the convention $0^0:=1$. Moreover, we introduce the notation $\{1:s\}:=\{1,2,\ldots,s\}$ and define the support of a multi-index $\bsnu$ by setting ${\rm supp}(\bsnu):=\{j\in \{1:s\}\mid \nu_j\neq 0\}$.

For \vk{a nonempty} domain $D\subset \mathbb R^d$ with given $d\in\mathbb N$, we define the Sobolev space $H^k(D)$ of order $k\in\mathbb N$ by
\[
H^k(D):=\{u\in L^2(D):\partial_{\bsx}^\bsalpha u \in L^2(D) \mbox{ for all }|\bsalpha|\le k,~\bsalpha\in\mathbb N_0^d \}
\]
and we equip this space with the norm
\[
\|u\|_{H^k(D)}:=\bigg(\sum_{|\bsalpha|\le k}\int_D |\partial_{\bsx}^\bsalpha u(\bsx)|^2\,\mathrm d\bsx\bigg)^{1/2},\quad u\in H^k(D),
\]
induced by the inner product
\[
\langle u,v \rangle_{H^k(D)}:=\sum_{|\bsalpha|\le k}\int_D \partial_{\bsx}^\bsalpha u(\bsx)\partial_{\bsx}^\bsalpha v(\bsx)\,\mathrm d\bsx,\quad u,v\in H^k(D).
\]
Further, $H^k_0(D)$ is the closure of $C^\infty_0(D)$ in the topology of $H^k(D)$.
\vk{We define $\|\boldsymbol x\|_{\Gamma^{-1}}:=\sqrt{\boldsymbol x^{\rm T}\Gamma^{-1}\boldsymbol x}$ for $\bsx\in \mathbb R^d$ and a symmetric positive definite matrix $\Gamma\in\mathbb R^{d \times d}$}.

\section{Problem setting}\label{sec:problem}
Let $G_s\!:{\Theta_s}\times \Xi\to \mathbb R^k$ be a mapping depending on parameter $\bstheta\in{\Theta_s}$ and a design parameter $\bsxi\in\Xi$. %
 \cs{For simplicity, we assume in the following that both the parameter space ${\Theta_s}$ and the design space $\Xi$ are finite-dimensional, compact subsets of Euclidean spaces, possibly obtained after dimension truncation. For forward problems governed by (partial) differential equations with random fields as unknown parameters, the truncation error is well understood, cf.~\cite{doi:10.1137/23M1593188}. The dependence on the dimension of the truncated parameter domain will be carefully tracked in this manuscript in order to design a method suitable for the high or even infinite-dimensional setting.}

We consider the measurement model
\begin{align}
\bsy=G_s(\bstheta,\bsxi)+\bseta,\label{eq:model}
\end{align}
where $\bsy\in\mathbb R^k$ is the measurement data and $\bseta\in\mathbb R^k$ is Gaussian noise such that $\bseta\sim \mathcal N(\mathbf 0,\Gamma)$, with $\Gamma \in \mathbb R^{k\times k}$ being symmetric positive definite.

In Bayesian optimal experimental design, the goal is to recover the design parameter $\bsxi$ for the Bayesian inference of $\bstheta$, which we model as a random variable endowed with a prior distribution $\pi(\bstheta)$.  A measure of the information gain for a given design $\bsxi$ and data $\bsy$ is given by the \emph{Kullback--Leibler divergence}
\begin{align}
D_{\rm KL}(\pi(\cdot|\bsy,\bsxi)\|\pi(\cdot)):=\int_{{\Theta_s}}\log\bigg(\frac{\pi(\bstheta|\bsy,\bsxi)}{\pi(\bstheta)}\bigg)\pi(\bstheta|\bsy,\bsxi)\,{\rm d}\bstheta.\label{eq:KLdef}
\end{align}
Assuming existence and uniqueness, a Bayesian optimal design $\bsxi^*$ maximizing the expected utility~\eqref{eq:KLdef} over the design space $\Xi$ with respect to the data $\bsy$ and model parameters $\bstheta$ is then given by
\begin{align}
\bsxi^*:=\underset{\bsxi\in\Xi}{\rm arg\,max}\int_{\rev{\mathbb R^k}}\int_{{\Theta_s}}\log\bigg(\frac{\pi(\bstheta|\bsy,\bsxi)}{\pi(\bstheta)}\bigg)\pi(\bstheta|\bsy,\bsxi)\pi(\bsy|\bsxi)\,{\rm d}\bstheta\,{\rm d}\bsy,\label{eq:KLdef2}
\end{align}
where $\pi(\bstheta|\bsy,\bsxi)$ corresponds to the posterior distribution of the parameter $\bstheta$ and $\pi(\bsy|\bsxi):=\int_{{\Theta_s}}\pi(\bsy|\bstheta,\bsxi)\pi(\bstheta)\,{\rm d}\bstheta$ is the marginal distribution of the data $\bsy$. The posterior is given by Bayes' theorem
$$
\pi(\bstheta|\bsy,\bsxi)=\frac{\pi(\bsy|\bstheta,\bsxi)\pi(\bstheta)}{\pi(\bsy|\bsxi)},
$$
where we have the data likelihood
$$
\pi(\bsy|\bstheta,\bsxi):=C_{k,\Gamma}{\rm e}^{-\frac{1}{2} \|\bsy-G_s(\bstheta,\bsxi)\|_{\Gamma^{-1}}^2},\quad C_{k,\Gamma}:=\frac{1}{(\det(2\pi\Gamma))^{1/2}}.
$$
The {\em expected information gain} is the objective appearing in~\eqref{eq:KLdef2}, i.e.,
$$
{\rm EIG}:=\int_{\rev{\mathbb R^k}}\int_{{\Theta_s}}\log\bigg(\frac{\pi(\bstheta|\bsy,\bsxi)}{\pi(\bstheta)}\bigg)\pi(\bstheta|\bsy,\bsxi)\pi(\bsy|\bsxi)\,{\rm d}\bstheta\,{\rm d}\bsy.
$$
The goal of this paper is to develop a rigorous framework within which the high-dimensional integrals appearing in EIG can be approximated efficiently using QMC methods.

\section{Quasi-Monte Carlo integration}\label{sec:qmc}
Consider an $s$-dimensional integration problem
$$
I_s(F):=\int_{[0,1]^s}F(\boldsymbol{y})\,{\rm d}\boldsymbol{y}
$$
\cs{with continuous $F\!:[0,1]^s\to \mathbb R$}.
A randomly shifted lattice rule is a QMC cubature of the form
$$
\overline{Q}_{n,R}(F):=\frac{1}{R}\sum_{r=1}^R Q(F;\boldsymbol \Delta^{(r)}),\quad Q(F;\boldsymbol\Delta^{(r)}):=\frac{1}{n}\sum_{i=1}^nF(\{\boldsymbol t_i+\boldsymbol\Delta^{(r)}\}),
$$
where $\boldsymbol\Delta^{(1)},\ldots,\boldsymbol\Delta^{(R)}$ are i.i.d.~random shifts drawn from $\mathcal U([0,1]^s)$, $\{\cdot\}$ denotes the componentwise fractional part, the lattice points are
\begin{align}
\boldsymbol t_i:=\bigg\{\frac{i\boldsymbol z}{n}\bigg\},\quad i\in\{1,\ldots,n\},\label{eq:latticepoint}
\end{align}
and $\boldsymbol z\in\{0,\ldots,n-1\}^s$ denotes the \emph{generating vector}.

Suppose that the integrand $F$ belongs to a weighted unanchored Sobolev space $\mathcal W_{{s,\boldsymbol\gamma}}$ with bounded first order mixed partial derivatives with the norm
$$
\|F\|_{{\mathcal W_{s,\boldsymbol\gamma}}}:=\bigg(\sum_{\mathfrak u\subseteq\{1:s\}}\frac{1}{\gamma_{\mathfrak u}}\int_{[0,1]^{|\mathfrak u|}}\bigg(\int_{[0,1]^{s-|\mathfrak u|}}\frac{\partial^{|\mathfrak u|}}{\partial \boldsymbol{y}_{\mathfrak u}}F(\boldsymbol{y})\,{\rm d}\boldsymbol{y}_{-\mathfrak u}\bigg)^2\,{\rm d}\boldsymbol{y}_{\mathfrak u}\bigg)^{1/2},
$$
where $\boldsymbol\gamma:=(\gamma_{\mathfrak u})_{\mathfrak u\subseteq\{1:s\}}$ is a collection of positive weights, ${\rm d}\boldsymbol{y}_{\mathfrak u}:=\prod_{j\in\mathfrak u}{\rm d}y_j$, and ${\rm d}\boldsymbol{y}_{-\mathfrak u}:=\prod_{j\in\{1:s\}\setminus\mathfrak u}{\rm d}y_j$. Then the following well-known result shows that there exists a sequence of generating vectors which can be constructed using a \emph{component-by-component} (CBC) algorithm with rigorous error bounds~\cite{ckn06,dks13,cn06}.
\begin{lemma}[cf.~{\cite[Theorem~5.1]{kuonuyenssurvey}}]\label{lemma:affineqmc}
Let $F$ belong to the weighted unanchored Sobolev space over $[0,1]^s$ with weights $\boldsymbol{\gamma}=(\gamma_{\mathfrak u})_{\mathfrak u\subseteq\{1:s\}}$. An $s$-dimensional randomly shifted lattice rule with $n=2^m$ points, $m\geq 0$, can be constructed by a CBC algorithm such that, for $R$ independent random shifts and for all $\lambda\in (1/2,1]$,
$$
\textup{R.M.S.~error}\leq\frac{1}{\sqrt{R}}\bigg(\frac{2}{n}\sum_{\varnothing\neq \mathfrak u\subseteq\{1:s\}}\gamma_{\mathfrak u}^\lambda\varrho(\lambda)^{|\mathfrak u|}\bigg)^{1/(2\lambda)}\|F\|_{\mathcal W_{s,\boldsymbol\gamma}},
$$
where \textup{R.M.S.~error}$\,:=\sqrt{\mathbb{E}_{\boldsymbol{\Delta}}[|I_{s}(F)-\overline{Q}_{n,R}(F)|^2]}$ and 
$$
\varrho(\lambda):=\frac{2\zeta(2\lambda)}{(2\pi^2)^\lambda}.
$$
Here, $\mathbb E_{\boldsymbol\Delta}[\cdot]$ denotes the expected value with respect to uniformly distributed random shift over $[0,1]^s$ and $\zeta(x):=\sum_{\ell=1}^\infty \ell^{-x}$ is the \emph{Riemann zeta function} for $x>1$.
\end{lemma}

\section{Model problem}\label{sec:model}
We make the following assumptions regarding the properties of the mathematical model {\eqref{eq:model}}.

\begin{assumption}\rm $ $
\begin{enumerate}[label={\rm (A1.\arabic*)},align=right,leftmargin=2.5\parindent]%
\item \cs{${\Theta_s}:=[-\frac12,\frac12]^s$} and $\pi(\bstheta):=1$ for $\bstheta\in{\Theta_s}$ and $0$ otherwise.\label{eq:A2}
\item\label{eq:A1} \vk{Let $C\geq 1$ \rev{and $\beta\geq1$ be constants} and let $\bsb:=(b_j)_{j=1}^\infty\in\ell^p(\mathbb N)$ for some $p\in(0,1)$ be a sequence of nonnegative real numbers independently of $s\in \mathbb N$ and $\bsxi\in\Xi$. There} exists a sequence of forward models \vk{$G_s\!:{\Theta_s}\times\Xi\to \mathbb R^k$}, indexed by $s\in\mathbb N$, which satisfy the parametric regularity bound
$$
\|\partial_{\bstheta}^{\bsnu}G_s(\bstheta,\bsxi)\|\leq C\rev{(|\bsnu|!)^{\beta}}\bsb^{\bsnu}%
$$
for all $\bstheta\in{\Theta_s}$, $\bsnu\in\mathbb N_0^s$, and $\bsxi\in\Xi$.
\item There exists a lower bound $0<\mu_{\min}\leq1$ on the smallest eigenvalue of~$\Gamma$.\label{eq:A3}
\end{enumerate}
\end{assumption}

Assumption~\ref{eq:A1} \rev{corresponds to a \emph{Gevrey regularity} assumption. There has been a lot of interest in the study of parametric PDE problems subject to Gevrey regular input random fields recently~\cite{chernovle2,chernovle1,schmidlin24}. Specifically, it has been observed to be a sufficient condition to establish dimension-independent QMC convergence rates,} while condition~\ref{eq:A2} implies that
\begin{align}
{\rm EIG}&=\int_{\rev{\mathbb R^k}}\int_{{\Theta_s}}\log\bigg(\frac{\pi(\boldsymbol y|\boldsymbol\theta,\boldsymbol\xi)}{\pi(\boldsymbol y|\boldsymbol\xi)}\bigg)\pi(\boldsymbol y|\boldsymbol\theta,\boldsymbol\xi)\pi(\boldsymbol\theta)\,{\rm d}\boldsymbol\theta\,{\rm d}\boldsymbol y\notag\\
&=\log C_{k,\Gamma}-\frac{k}{2} -\int_{\rev{\mathbb R^k}} \log\bigg(\int_{{\Theta_s}}C_{k,\Gamma}{\rm e}^{-\Phi(\boldsymbol\theta,\boldsymbol\xi)}\,{\rm d}\boldsymbol\theta\bigg)\int_{{\Theta_s}}C_{k,\Gamma}{\rm e}^{-\Phi(\boldsymbol\theta,\boldsymbol\xi)}\,{\rm d}\boldsymbol\theta\,{\rm d}\boldsymbol y,\label{eq:numexref}
\end{align}
with potential $\Phi(\boldsymbol\theta,\boldsymbol\xi)=\frac{1}{2} \|\boldsymbol y-G_s(\boldsymbol\theta,\boldsymbol\xi)\|_{\Gamma^{-1}}^2$, meaning that it suffices to investigate the double integral
\begin{align}\label{eq:integralofinterest}
\int_{\rev{\mathbb R^k}} \log\bigg(\int_{{\Theta_s}}C_{k,\Gamma}{\rm e}^{-\frac{1}{2} \|\boldsymbol y-G_s(\boldsymbol\theta,\boldsymbol\xi)\|_{\Gamma^{-1}}^2}\,{\rm d}\boldsymbol\theta\bigg)\int_{{\Theta_s}}C_{k,\Gamma}{\rm e}^{-\frac{1}{2} \|\boldsymbol y-G_s(\boldsymbol\theta,\boldsymbol\xi)\|_{\Gamma^{-1}}^2}\,{\rm d}\boldsymbol\theta\,{\rm d}\boldsymbol y.\end{align}

\vk{We note that the above set of assumptions also cover elliptic PDE problems subject to uncertain coefficients.

\begin{example}\rm
Let $D\subset \mathbb R^d$, $d\in\{1,2,3\}$, be a nonempty, bounded, and convex Lipschitz domain and let $z\in L^2(D)$. For each $\bstheta\in {\Theta}:=[-1/2,1/2]^{\mathbb N}$, there exists a strong solution $u(\cdot,\bstheta)\in H^2(D)\cap H_0^1(D)$ to
$$
\begin{cases}
-\nabla \cdot (a(\bsx,\bstheta)\nabla u(\bsx,\bstheta))=z(\bsx),&\bsx\in D,~\bstheta\in{\Theta_s},\\
u(\bsx,\bstheta)=0,&\bsx\in\partial D,~\bstheta\in{\Theta_s},
\end{cases}
$$
where the diffusion coefficient is parameterized by
$$
a(\bsx,\bstheta)=a_0(\bsx)+\sum_{j=1}^\infty\theta_j\psi_j(\bsx),\quad \bsx\in D,~\bstheta\in {\Theta},
$$
with $a_0\!:\overline{D}\to\mathbb R$ and $\psi_j\!:\overline{D}\to\mathbb R$, $j\in\{1,\ldots,s\}$, denoting Lipschitz continuous functions such that $0<a_{\min}\leq a(\bsx,\bstheta)\leq a_{\max}<\infty$ for all $\bsx\in D$ and $\bstheta\in {\Theta}$.

Let $G_s(\bstheta,\bsxi):=\mathcal O_{\bsxi}(u(\cdot,(\theta_1,\ldots,\theta_s,0,0,\ldots)))$, $\bstheta\in\Theta_s$, with $\mathcal O_{\bsxi}\!:H_0^1(D)\to \mathbb R^k$ denoting a bounded, linear functional such that $\sup_{\bsxi\in \Xi}\|\mathcal O_{\bsxi}\|_{H_0^1\to \mathbb R}<\infty$. An example of such an operator would be, e.g., $\mathcal O_{\bsxi}(v)=(v(\bsx))_{\bsx\in \bsxi}$ for $\bsxi\in \Xi$ with $\Xi=\{(\bsx_1,\ldots,\bsx_k)\in\Upsilon^k\mid \bsx_i\neq \bsx_j~\text{for}~i\neq j\}$, where $\Upsilon\subset D$ and $0<k\leq m:=|\Upsilon|<\infty$. In this case, the optimal design problem~\eqref{eq:KLdef2} would correspond to choosing the best $k$ sensor locations out of $m$ possibilities to maximize the expected information gain on the unknown parameter $\bstheta$.
\end{example}}

\vk{While we shall mainly focus on applying QMC integration over lattice point sets to~\eqref{eq:integralofinterest} subject to~\ref{eq:A2}--\ref{eq:A3}, it is well-known that lattice point sets yield higher-order cubature convergence rates for periodic integrands (cf., e.g.,~\cite{korobovpaper}). In analogy to~\cite{KKS}, we shall also study the EIG for a periodic reparameterization of our model problem
\begin{align}
G_{s,{\rm per}}(\bstheta)=G_s(\sin(2\pi\bstheta)),\quad \bstheta\in \Theta_s.\label{eq:periodicmodelproblem}
\end{align}
We summarize our results for the parametric regularity and QMC integration rates corresponding to the model~\eqref{eq:periodicmodelproblem} in Appendix~\ref{sec:periodic}.}

\subsection{Decomposing the high-dimensional integral}

For the practical implementation and subsequent analysis, we let $K>0$ and decompose
\begin{align*}
&\int_{\mathbb R^k}\log\bigg(\int_{\Theta_s} C_{k,\Gamma}{\rm e}^{-\frac{1}{2}\|\bsy-G_s(\bstheta,\bsxi)\|_{\Gamma^{-1}}^2}\,{\rm d}\bstheta\bigg)\int_{\Theta_s} C_{k,\Gamma}{\rm e}^{-\frac{1}{2}\|\bsy-G_s(\bstheta,\bsxi)\|_{\Gamma^{-1}}^2}\,{\rm d}\bstheta\,{\rm d}\bsy\\
&=\mathcal I_K + \widetilde{\mathcal I}_K,
\end{align*}
where
\begin{align*}
\mathcal I_K :\!&\!=\int_{[-K,K]^k}\log\bigg(\int_{\Theta_s} C_{k,\Gamma}{\rm e}^{-\frac{1}{2}\|\bsy-G_s(\bstheta,\bsxi)\|_{\Gamma^{-1}}^2}\,{\rm d}\bstheta\bigg)\\
&\quad \times\int_{\Theta_s} C_{k,\Gamma}{\rm e}^{-\frac{1}{2}\|\bsy-G_s(\bstheta,\bsxi)\|_{\Gamma^{-1}}^2}\,{\rm d}\bstheta\,{\rm d}\bsy,\\
\widetilde{\mathcal I}_K:\!&\!=\int_{\mathbb R^k\setminus[-K,K]^k}\log\bigg(\int_{\Theta_s} C_{k,\Gamma}{\rm e}^{-\frac{1}{2}\|\bsy-G_s(\bstheta,\bsxi)\|_{\Gamma^{-1}}^2}\,{\rm d}\bstheta\bigg)\\
&\quad \times \int_{\Theta_s} C_{k,\Gamma}{\rm e}^{-\frac{1}{2}\|\bsy-G_s(\bstheta,\bsxi)\|_{\Gamma^{-1}}^2}\,{\rm d}\bstheta\,{\rm d}\bsy.
\end{align*}
Let us analyze the quantity $\widetilde{\mathcal I}_K$. %
\cs{First of all, we observe that
\begin{align*}
\int_{\Theta_s} C_{k,\Gamma}{\rm e}^{-\frac{1}{2}\|\bsy-G_s(\bstheta,\bsxi)\|_{\Gamma^{-1}}^2}\,{\rm d}\bstheta&\leq C_{k,\Gamma}{\rm e}^{\frac{2\|\overline{G}_s\|^2}{\mu_{\max}}}{\rm e}^{-\frac{1}{4\mu_{\max}}\|\bsy-\overline{G}_s\|^2}\quad\text{for all}~\bsy\in \mathbb R^k,
\end{align*}
where $\overline{G}_s:=G_s(\bstheta^*,\bsxi^*)$, with $(\bstheta^*,\bsxi^*)\in {\rm arg\,max}_{(\bstheta,\bsxi)\in (\Theta_s,\Xi)}\|G_s(\bstheta,\bsxi)\|$, and $\mu_{\max}>0$ is the largest eigenvalue of $\Gamma$.
Since $x\mapsto |x\log x|$ is a monotonically increasing function over the interval $(0,{\rm e}^{-1}]$, it is not difficult to see that setting
$$
K\geq \|\overline{G}_s\|+2\sqrt{\mu_{\max}\log C_{k,\Gamma}+2\|\overline{G}_s\|^2+\mu_{\max}}
$$
ensures that $C_{k,\Gamma}{\rm e}^{\frac{2\|\overline{G}_s\|^2}{\mu_{\max}}}{\rm e}^{-\frac{1}{4\mu_{\max}}\|\bsy-\overline{G}_s\|^2}\leq {\rm e}^{-1}$ for all $\bsy\in\mathbb R^k\setminus[-K,K]^k$. This allows us to estimate
\begin{align*}
&|\widetilde{\mathcal I}_K|\\
&\!\leq \!-\!\int_{\mathbb R^k\setminus[-K,K]^k}\!\log(C_{k,\Gamma}{\rm e}^{\frac{2\|\overline{G}_s\|^2}{\mu_{\max}}}\!{\rm e}^{-\frac{1}{4\mu_{\max}}\|\bsy-\overline{G}_s\|^2})C_{k,\Gamma}{\rm e}^{\frac{2\|\overline{G}_s\|^2}{\mu_{\max}}}\!{\rm e}^{-\frac{1}{4\mu_{\max}}\|\bsy-\overline{G}_s\|^2}{\rm d}\bsy.
\end{align*}
Noting further that $0\leq -x\log x\leq \sqrt x$ for all $x\in(0,{\rm e}^{-1}]$, we deduce that
\begin{align*}
&|\widetilde{\mathcal I}_K|\\
&\leq \!C_{k,\Gamma}^{1/2}{\rm e}^{\frac{\|\overline{G}_s\|^2}{\mu_{\max}}}\!\int_{\mathbb R^k}{\rm e}^{-\frac{1}{8\mu_{\max}}\|\bsy-\overline{G}_s\|^2}{\rm d}\bsy-C_{k,\Gamma}^{1/2}{\rm e}^{\frac{\|\overline{G}_s\|^2}{\mu_{\max}}}\int_{[-K,K]^k}{\rm e}^{-\frac{1}{8\mu_{\max}}\|\bsy-\overline{G}_s\|^2}{\rm d}\bsy\\
&=\!C_{k,\Gamma}^{1/2}{\rm e}^{\frac{\|\overline{G}_s\|^2}{\mu_{\max}}}(8\pi\mu_{\max})^{k/2}\\
&\quad \  -C_{k,\Gamma}^{1/2}{\rm e}^{\frac{\|\overline{G}_s\|^2}{\mu_{\max}}}(8\pi\mu_{\max})^{k/2}\prod_{j=1}^k\frac12\bigg({\rm erf}\bigg(\frac{(\overline{G}_s)_j+K}{\sqrt{8\mu_{\max}}}\bigg)-{\rm erf}\bigg(\frac{(\overline{G}_s)_j-K}{\sqrt{8\mu_{\max}}}\bigg)\bigg).
\end{align*}
}
\rev{
\begin{remark}
Given a prescribed tolerance $\varepsilon>0$, we may choose $K$ such that $|\widetilde{\mathcal I}_K|\le \varepsilon$. 
Indeed, setting $M:=\max_{1\le j\le k}|(\overline G_s)_j|$, we obtain for $K\ge M$
\begin{align*}
|\widetilde{\mathcal I}_K|
&\le 
C_{k,\Gamma}^{1/2}{\rm e}^{\frac{\|\overline{G}_s\|^2}{\mu_{\max}}}\int_{\mathbb R^k\setminus[-K,K]^k}
    {\rm e}^{-\frac{1}{8\mu_{\max}}\|\bsy-\overline G_s\|^2}\,{\rm d}\bsy \\
&\le 
C_{k,\Gamma}^{1/2}{\rm e}^{\frac{\|\overline{G}_s\|^2}{\mu_{\max}}}\,{\rm e}^{-\frac{(K-M)^2}{8\mu_{\min}}}.
\end{align*}
Hence, it suffices to take
\[
K \ge M + 2\sqrt{2\mu_{\max}\log(C_{k,\Gamma}^{1/2}{\rm e}^{\frac{\|\overline{G}_s\|^2}{\mu_{\max}}}/\varepsilon)}.
\]
\end{remark}}
\section{Parametric regularity}\label{sec:reg}
In order to establish dimension-independent QMC convergence rates, we first analyze the parametric regularity of the integrand. The analysis is split into the inner \vk{(parametric)} integrand (Subsection~\ref{sec:inner}), the mixed regularity  (Subsection~\ref{sec:mixed}), and the outer \vk{(data)} integrand (Subsection~\ref{sec:outer}).
\subsection{Parametric regularity of the inner integrand}\label{sec:inner}
We begin by considering the parametric regularity of the inner integrand appearing in the expression
\begin{align}
\int_{{\Theta_s}}{\rm e}^{-\frac{1}{2} \|\boldsymbol y-G_s(\boldsymbol\theta,\boldsymbol\xi)\|_{\Gamma^{-1}}^2}\,{\rm d}\boldsymbol\theta.\label{eq:innerintdef}
\end{align}

Cram\'er's inequality (cf., e.g.,~\cite[formula 22.14.17]{abramowitzstegun}) 
$$
\bigg|\frac{{\rm d}^{\nu}}{{\rm d}x^{\nu}}{\rm e}^{-x^2/2}\bigg|\leq 1.1\cdot \sqrt{{\nu}!}\quad\text{for all}~\nu\geq 0~\text{and}~x\in\mathbb R
$$
yields that
$$
|\partial_{\boldsymbol x}^{\boldsymbol{\nu}}{\rm e}^{-\boldsymbol{x}^{\rm T}\boldsymbol{x}/2}|\leq 1.1^k\cdot \sqrt{\boldsymbol{\nu}!}\quad\text{for all}~\boldsymbol\nu\in\mathbb N_0^k~\text{and}~\boldsymbol x\in\mathbb R^k.
$$
Let $\boldsymbol{\nu}\in \mathbb N_0^d\setminus\{\mathbf 0\}$. By Fa\`a di Bruno's formula~\cite{savits}, we have that
\begin{align}
\partial_{\boldsymbol\theta}^{\boldsymbol{\nu}}{\rm e}^{-\frac{1}{2} \|\boldsymbol y-G_s(\boldsymbol\theta,\boldsymbol\xi)\|_{\Gamma^{-1}}^2}&=\sum_{\substack{\boldsymbol \lambda\in \mathbb N_0^k\\ 1\leq |\boldsymbol \lambda|\leq |\boldsymbol{\nu}|}}\partial_{\boldsymbol x}^{\boldsymbol\lambda}{\rm e}^{-\boldsymbol x^{\rm T}\boldsymbol x/2}\bigg|_{\boldsymbol x=\Gamma^{-1/2}(\boldsymbol y-G_s(\boldsymbol\theta,\boldsymbol\xi))} \alpha_{\boldsymbol{\nu},\boldsymbol\lambda}(\boldsymbol\theta),\label{eq:faadibruno}
\end{align}
where $(\alpha_{\boldsymbol{\nu},\boldsymbol\lambda})_{\boldsymbol{\nu}\in\mathbb N_0^s,\boldsymbol\lambda\in\mathbb Z^k}=(\alpha_{\boldsymbol{\nu},\boldsymbol\lambda}(\bstheta))_{\boldsymbol{\nu}\in\mathbb N_0^s,\boldsymbol\lambda\in\mathbb Z^k}$ are defined recursively by
\begin{align}\label{eq:alphaseq1}
&\alpha_{\boldsymbol{\nu},\mathbf 0}\equiv \delta_{\boldsymbol{\nu},\mathbf 0},\\
&\alpha_{\boldsymbol{\nu},\boldsymbol\lambda}\equiv 0\quad\text{if}~|\boldsymbol{\nu}|<|\boldsymbol\lambda|~\text{or}~\boldsymbol\lambda\not\geq \mathbf 0~(\text{i.e.,~if $\boldsymbol\lambda$ contains negative entries}),~\text{and}\label{eq:alphaseq2}\\
&\alpha_{\boldsymbol{\nu}+\boldsymbol e_j,\boldsymbol\lambda}= \sum_{\ell\in{\rm supp}(\boldsymbol\lambda)} \sum_{\mathbf 0\leq \boldsymbol{m}\leq \boldsymbol{\nu}}\binom{\boldsymbol{\nu}}{\boldsymbol{m}}\partial_{\bstheta}^{\boldsymbol{m}+\boldsymbol e_j}[\Gamma^{-1/2}(\boldsymbol y- G_s(\boldsymbol\theta,\boldsymbol\xi))]_\ell\alpha_{\boldsymbol{\nu}-\boldsymbol{m},\boldsymbol\lambda-\boldsymbol e_\ell}\label{eq:alphaseq3}
\end{align}
otherwise. From this and the assumption~\ref{eq:A1}, we easily infer that
\begin{align*}
|\partial_{\boldsymbol\theta}^{\boldsymbol\nu}{\rm e}^{-\frac{1}{2}\|\boldsymbol y-G_s(\boldsymbol\theta,\boldsymbol\xi)\|_{\Gamma^{-1}}^2}|&\leq 1.1^k\cdot C^{|\boldsymbol{\nu}|}\mu_{\min}^{-|\boldsymbol{\nu}|/2}\sum_{\substack{\boldsymbol \lambda\in \mathbb N_0^k\\ 1\leq |\boldsymbol \lambda|\leq |\boldsymbol{\nu}|}}\sqrt{\boldsymbol\lambda!}\rev{\chi}_{\boldsymbol{\nu},\boldsymbol\lambda},
\end{align*}
where $0<\mu_{\min}\leq 1$ by assumption~\ref{eq:A3}. We now have the recurrence
\begin{align}
&\rev{\chi}_{\boldsymbol{\nu},\mathbf 0}= \delta_{\boldsymbol{\nu},\mathbf 0},\label{eq:faa1}\\
&\rev{\chi}_{\boldsymbol{\nu},\boldsymbol\lambda}= 0\quad\text{if}~|\boldsymbol{\nu}|<|\boldsymbol\lambda|~\text{or}~\boldsymbol\lambda\not\geq \mathbf 0,\label{eq:faa2}\\
&\rev{\chi}_{\boldsymbol{\nu}+\boldsymbol e_j,\boldsymbol\lambda}\leq \sum_{\ell\in{\rm supp}(\boldsymbol\lambda)} \sum_{\mathbf 0\leq \boldsymbol{m}\leq \boldsymbol{\nu}}\binom{\boldsymbol{\nu}}{\boldsymbol{m}}\rev{((|\boldsymbol{m}|+1)!)^{\beta}}\boldsymbol b^{\boldsymbol{m}+\boldsymbol e_j}\rev{\chi}_{\boldsymbol{\nu}-\boldsymbol{m},\boldsymbol\lambda-\boldsymbol e_\ell}~\text{otherwise}.\label{eq:faa3}
\end{align}
 
The above recursion leads to the following inductive bound.
\begin{lemma}\label{lemma:lah} Let $\boldsymbol{\nu}\in\mathbb N_0^s\setminus\{\mathbf 0\}$ and $\boldsymbol\lambda\in\mathbb N_0^k\setminus\{\mathbf 0\}$ such that $1\leq |\boldsymbol\lambda|\leq |\boldsymbol{\nu}|$. The sequence defined by~\eqref{eq:faa1}--\eqref{eq:faa3} satisfies%
\begin{align}
\rev{\chi}_{\boldsymbol{\nu},\boldsymbol\lambda}\leq\rev{\bigg(\frac{|\boldsymbol{\nu}|!(|\boldsymbol{\nu}|-1)!}{\boldsymbol\lambda!(|\boldsymbol{\nu}|-|\boldsymbol\lambda|)!(|\boldsymbol\lambda|-1)!}\bigg)^{\beta}}\boldsymbol b^{\boldsymbol{\nu}}.\label{eq:lah}
\end{align}
\rev{If $\beta=1$, then the} result is sharp in the sense that~\eqref{eq:lah} holds with equality provided that~{\rm \eqref{eq:faa3}} holds with equality.
\end{lemma}
\begin{proof} The proof is carried out using induction with respect to the order of $\boldsymbol{\nu}$. 

The base step $|\boldsymbol{\nu}|=1$ is satisfied since for $j\in\{1,\ldots,s\}$ and $k'\in\{1,\ldots,k\}$, we have by~\eqref{eq:faa3} that
\begin{align*}
\rev{\chi}_{\boldsymbol e_j,\boldsymbol e_{k'}}\leq \sum_{\ell\in{\rm supp}(\boldsymbol\lambda)} \boldsymbol b^{\boldsymbol e_j}\rev{\chi}_{\mathbf 0,\boldsymbol e_{k'}-\boldsymbol e_\ell}=\boldsymbol b^{\boldsymbol e_j},
\end{align*}
as desired.

Let $\boldsymbol{\nu}\in\mathbb N_0^s\setminus\{\mathbf 0\}$ and suppose that the claim is true for all multi-indices with order $\leq |\boldsymbol{\nu}|$. We wish to prove the claim for all $1\leq |\boldsymbol\lambda|\leq |\boldsymbol{\nu}|+1$.

Let us begin by considering the special case $|\boldsymbol\lambda|=1$ separately. For $j\in\{1,\ldots,s\}$ and $k'\in\{1,\ldots,k\}$, we have that
\begin{align*}
\rev{\chi}_{\boldsymbol{\nu}+\boldsymbol e_j,\boldsymbol e_{k'}}&\leq \sum_{\ell\in{\rm supp}(\boldsymbol\lambda)}\sum_{\mathbf 0\leq \boldsymbol m\leq \boldsymbol{\nu}}\binom{\boldsymbol{\nu}}{\boldsymbol{m}}\rev{((|\boldsymbol{m}|+1)!)^{\beta}}\boldsymbol b^{\boldsymbol{m}+\boldsymbol e_j}\rev{\chi}_{\boldsymbol{\nu}-\boldsymbol{m},\boldsymbol e_{k'}-\boldsymbol e_{\ell}}\\
&= \sum_{\mathbf 0\leq \boldsymbol m\leq \boldsymbol{\nu}}\binom{\boldsymbol{\nu}}{\boldsymbol{m}}\rev{((|\boldsymbol{m}|+1)!)^{\beta}}\boldsymbol b^{\boldsymbol{m}+\boldsymbol e_j}\rev{\chi}_{\boldsymbol{\nu}-\boldsymbol{m},\mathbf 0}\\
&=\rev{((|\boldsymbol{\nu}|+1)!)^{\beta}}\boldsymbol b^{\boldsymbol{\nu}+\boldsymbol e_j},
\end{align*}
as desired.

Let $2\leq |\boldsymbol\lambda|\leq |\boldsymbol{\nu}|+1$. By noting that~\eqref{eq:faa1}--\eqref{eq:faa2} allow us to impose the restriction $|\boldsymbol{m}|\leq |\boldsymbol{\nu}|-|\boldsymbol\lambda|+1$, we use the induction hypothesis and the recurrence~\eqref{eq:faa3} to obtain
\begin{align*}
&\rev{\chi}_{\boldsymbol{\nu}+\boldsymbol e_j,\boldsymbol\lambda}\\
&\leq \!\!\sum_{\ell\in{\rm supp}(\boldsymbol\lambda)}\!\!\sum_{\substack{\mathbf 0\leq \boldsymbol{m}\leq\boldsymbol{\nu}\\ |\boldsymbol{m}|\leq |\boldsymbol{\nu}|-|\boldsymbol\lambda|+1}}\!\!\binom{\boldsymbol{\nu}}{\boldsymbol{m}}\boldsymbol b^{\boldsymbol{\nu}+\boldsymbol e_j}\rev{\bigg(\!\frac{(|\boldsymbol{m}|+1)!(|\boldsymbol{\nu}|-|\boldsymbol{m}|)!(|\boldsymbol{\nu}|-|\boldsymbol{m}|-1)!}{(\boldsymbol\lambda-\boldsymbol e_{\ell})!(|\boldsymbol{\nu}|-|\boldsymbol{m}|-|\boldsymbol\lambda|+1)!(|\boldsymbol\lambda|-2)!}\!\bigg)^{\!\beta}}\\
&=\rev{\bigg(\frac{|\boldsymbol\lambda|}{\boldsymbol\lambda!(|\boldsymbol\lambda|-2)!}\bigg)^{\beta}}\boldsymbol b^{\boldsymbol{\nu}+\boldsymbol e_j}\\
&\quad \times\sum_{\substack{\mathbf 0\leq \boldsymbol{m}\leq\boldsymbol{\nu}\\ |\boldsymbol{m}|\leq |\boldsymbol{\nu}|-|\boldsymbol\lambda|+1}}\binom{\boldsymbol{\nu}}{\boldsymbol{m}}\rev{\bigg(\frac{(|\boldsymbol{m}|+1)!(|\boldsymbol{\nu}|-|\boldsymbol{m}|)!(|\boldsymbol{\nu}|-|\boldsymbol{m}|-1)!}{(|\boldsymbol{\nu}|-|\boldsymbol{m}|-|\boldsymbol\lambda|+1)!}\bigg)^{\beta}}\\
&=\rev{\bigg(\frac{|\boldsymbol\lambda|}{\boldsymbol\lambda!(|\boldsymbol\lambda|-2)!}\bigg)^{\beta}}\boldsymbol b^{\boldsymbol{\nu}+\boldsymbol e_j}\!\!\sum_{\ell=0}^{|\boldsymbol{\nu}|-|\boldsymbol\lambda|+1}\!\!\rev{\bigg(\frac{(\ell+1)!(|\boldsymbol{\nu}|-\ell)!(|\boldsymbol{\nu}|-\ell-1)!}{(|\boldsymbol{\nu}|-\ell-|\boldsymbol\lambda|+1)!}\bigg)^{\beta}}\!\!\sum_{\substack{\boldsymbol{m}\leq \boldsymbol{\nu}\\ |\boldsymbol{m}|=\ell}}\binom{\boldsymbol{\nu}}{\boldsymbol{m}}\\
&=\rev{\bigg(\frac{|\boldsymbol\lambda|}{\boldsymbol\lambda!(|\boldsymbol\lambda|-2)!}\bigg)^{\beta}}\boldsymbol b^{\boldsymbol{\nu}+\boldsymbol e_j}\sum_{\ell=0}^{|\boldsymbol{\nu}|-|\boldsymbol\lambda|+1}\rev{\bigg(\frac{(\ell+1)!(|\boldsymbol{\nu}|-\ell)!(|\boldsymbol{\nu}|-\ell-1)!}{(|\boldsymbol{\nu}|-\ell-|\boldsymbol\lambda|+1)!}\bigg)^{\beta}}\binom{|\boldsymbol{\nu}|}{\ell}\\
&\rev{\leq\rev{\bigg(\frac{|\boldsymbol\lambda|}{\boldsymbol\lambda!(|\boldsymbol\lambda|-2)!}\bigg)^{\beta}}\boldsymbol b^{\boldsymbol{\nu}+\boldsymbol e_j}\sum_{\ell=0}^{|\boldsymbol{\nu}|-|\boldsymbol\lambda|+1}\rev{\bigg(\frac{(\ell+1)!(|\boldsymbol{\nu}|-\ell)!(|\boldsymbol{\nu}|-\ell-1)!}{(|\boldsymbol{\nu}|-\ell-|\boldsymbol\lambda|+1)!}\bigg)^{\beta}}\binom{|\boldsymbol{\nu}|}{\ell}^{\beta}}\\
&=\rev{\bigg(\frac{|\boldsymbol{\nu}|!|\boldsymbol\lambda|}{\boldsymbol\lambda!(|\boldsymbol\lambda|-2)!}\bigg)^{\beta}}\boldsymbol b^{\boldsymbol{\nu}+\boldsymbol e_j}\sum_{\ell=0}^{|\boldsymbol{\nu}|-|\boldsymbol\lambda|+1}\rev{(\ell+1)^{\beta}\bigg(\frac{(|\boldsymbol{\nu}|-\ell-1)!}{(|\boldsymbol{\nu}|-\ell-|\boldsymbol\lambda|+1)!}\bigg)^\beta},
\end{align*}
where we used the generalized Vandermonde identity $\sum_{\boldsymbol{m}\leq\boldsymbol{\nu},~|\boldsymbol{m}|=\ell}\binom{\boldsymbol{\nu}}{\boldsymbol{m}}=\binom{|\boldsymbol{\nu}|}{\ell}$ on the \rev{third} to last line. Since Lemma~\ref{lemma:gosper} implies that
$$
\sum_{\ell=0}^{v-1}(\ell+1)\frac{(v-\ell-1)!}{(v-\ell-\lambda+1)!}=\!\sum_{\ell=\lambda-1}^v (v-\ell+1)\frac{(\ell-1)!}{(\ell-\lambda+1)!}=\frac{(v+1)!}{(v-\lambda+1)!\lambda(\lambda-1)}
$$
for all $v\geq 1$ and $2\leq \lambda\leq v+1$, \rev{we can use Jensen's inequality (cf., e.g.,~\cite[Theorem~19]{HardyLP53}) \begin{align}\label{eq:jensenineq}\sum_k a_k^{\beta}\leq \big(\sum_k a_k\big)^{\beta}\quad\text{for all}~a_k\geq 0,~\beta\geq 1,\end{align}to} obtain
\begin{align*}
\rev{\chi}_{\boldsymbol{\nu}+\boldsymbol e_j,\boldsymbol\lambda}&\rev{\leq \bigg(\frac{|\boldsymbol{\nu}|!|\boldsymbol\lambda|}{\boldsymbol\lambda!(|\boldsymbol\lambda|-2)!}\bigg)^{\beta}\boldsymbol b^{\boldsymbol{\nu}+\boldsymbol e_j}\bigg(\sum_{\ell=0}^{|\boldsymbol{\nu}|-|\boldsymbol\lambda|+1}(\ell+1)\frac{(|\boldsymbol{\nu}|-\ell-1)!}{(|\boldsymbol{\nu}|-\ell-|\boldsymbol\lambda|+1)!}\bigg)^\beta}\\
&= \rev{\bigg(\frac{|\boldsymbol{\nu}|!(|\boldsymbol{\nu}|+1)!}{\boldsymbol\lambda!(|\boldsymbol\lambda|-1)!(|\boldsymbol{\nu}|-|\boldsymbol\lambda|+1)!}\bigg)^{\beta}}\boldsymbol b^{\boldsymbol{\nu}+\boldsymbol e_j},
\end{align*}
as desired.\qed\end{proof}

\begin{remark}\rm The coefficients $\rev{\frac{|\boldsymbol{\nu}|!(|\boldsymbol{\nu}|-1)!}{\boldsymbol\lambda!(|\boldsymbol{\nu}|-|\boldsymbol\lambda|)!(|\boldsymbol\lambda|-1)!}}$ in the upper bound of~\eqref{eq:lah} are known in the literature as \emph{multivariate Lah numbers}~\cite{be20}.\end{remark}

It immediately follows that
\begin{align*}
\sum_{\substack{\boldsymbol \lambda\in \mathbb N_0^k\\ 1\leq |\boldsymbol \lambda|\leq |\boldsymbol{\nu}|}}\sqrt{\boldsymbol\lambda!}\rev{\chi}_{\boldsymbol{\nu},\boldsymbol\lambda}&\leq \rev{(|\boldsymbol{\nu}|!(|\boldsymbol{\nu}|-1)!)^{\beta}}\boldsymbol b^{\boldsymbol{\nu}}\sum_{\ell=1}^{|\boldsymbol{\nu}|}\frac{1}{\rev{((|\boldsymbol{\nu}|-\ell)!(\ell-1)!)^{\beta}}}\sum_{\substack{\boldsymbol\lambda\in\mathbb N_0^k\\ |\boldsymbol\lambda|=\ell}}\frac{1}{\sqrt{\boldsymbol\lambda!}}\\
&\leq \rev{(|\boldsymbol{\nu}|!(|\boldsymbol{\nu}|-1)!)^{\beta}}\boldsymbol b^{\boldsymbol{\nu}}\sum_{\ell=1}^{|\boldsymbol{\nu}|}\frac{1}{\rev{((|\boldsymbol{\nu}|-\ell)!(\ell-1)!)^{\beta}}}\bigg(\sum_{\lambda=0}^\infty \frac{1}{\sqrt{\lambda!}}\bigg)^k\\
&\rev{\,\leq 3.47^k\cdot (|\boldsymbol{\nu}|!(|\boldsymbol{\nu}|-1)!)^{\beta}\boldsymbol b^{\boldsymbol{\nu}}\bigg(\sum_{\ell=1}^{|\boldsymbol{\nu}|}\frac{1}{\rev{(|\boldsymbol{\nu}|-\ell)!(\ell-1)!}}\bigg)^{\beta}}\\
&=3.47^k\cdot 2^{\rev{\beta(|\boldsymbol{\nu}|-1)}} \rev{(|\boldsymbol{\nu}|!)^{\beta}}\boldsymbol b^{\boldsymbol{\nu}},
\end{align*}
where we made use of $\sum_{\lambda=0}^\infty\frac{1}{\sqrt{\lambda!}}=3.469506\ldots$\rev{, Jensen's inequality~\eqref{eq:jensenineq},} and the summation identity (see Lemma~\ref{lemma:celine} in the Appendix)
$$
\sum_{\ell=1}^v \frac{1}{(v-\ell)!(\ell-1)!}=\frac{2^{v-1}}{(v-1)!}.
$$
This leads us to conclude the following.
\begin{lemma}\label{lemma:innerreg}Let $\bsnu\in\mathbb N_0^s\setminus\{\mathbf0\}$. Then under assumptions~{\rm \ref{eq:A2}}--\,{\rm \ref{eq:A3}}, there holds
$$
|\partial_{\boldsymbol\theta}^{\boldsymbol\nu}{\rm e}^{-\frac{1}{2}\|\boldsymbol y-G_s(\boldsymbol\theta,\boldsymbol\xi)\|_{\Gamma^{-1}}^2}|\leq 3.82^k\cdot C^{|\boldsymbol{\nu}|}2^{\rev{\beta(|\boldsymbol{\nu}|-1)}}\mu_{\min}^{-|\boldsymbol{\nu}|/2}\rev{(|\boldsymbol{\nu}|!)^{\beta}}\boldsymbol b^{\boldsymbol{\nu}}
$$
for all $\bstheta\in \Theta_s$, $\bsy\in [-K,K]^k$, and $\bsxi\in \Xi$.
\end{lemma}
\subsection{Mixed regularity}\label{sec:mixed}
For the analysis of the double integral, we will also need a parametric regularity bound of
$$
\partial_{\bsy}^{\bseta}\partial_{\bstheta}^{\bsnu}{\rm e}^{-\frac{1}{2}\|\bsy-G_s(\bstheta,\bsxi)\|_{\Gamma^{-1}}^2}\quad\text{for}~\bseta\in\mathbb N_0^k~\text{and}~\bsnu\in\mathbb N_0^s.
$$
It is sufficient that this term is bounded uniformly for all $(\bsy,\bstheta)\in \cs{[-K,K]}^k\times \Theta_s$. By \eqref{eq:faadibruno}, we have
\begin{align}\label{eq:mixedderivatives}
\partial_{\bsy}^{\bseta}\partial_{\bstheta}^{\bsnu}{\rm e}^{-\frac{1}{2}\|\bsy-G_s(\bstheta,\bsxi)\|_{\Gamma^{-1}}^2}=\sum_{\substack{\bslambda\in\mathbb N_0^k\\1\leq|\bslambda|\leq |\bsnu|}}\partial_{\bsy}^{\bseta}\partial_{\bsx}^{\bslambda}{\rm e}^{-\bsx^{\rm T}\bsx/2}\bigg|_{\bsx=\Gamma^{-1/2}(\bsy-G_s(\bstheta,\bsxi))}\alpha_{\bsnu,\bslambda}(\bstheta),
\end{align}
where the sequence $(\alpha_{\bsnu,\bslambda}(\bstheta))$ is defined by~\eqref{eq:alphaseq1}--\eqref{eq:alphaseq3}. Making use of the formula
$$
\frac{{\rm d}^{\lambda}}{{\rm d}x^{\lambda}}{\rm e}^{-\frac{x^2}{2}}=\sum_{m=0}^{\lfloor \lambda/2\rfloor}\frac{(-1)^{m+\lambda}\lambda!}{m!(\lambda-2m)!2^m}x^{\lambda-2m}{\rm e}^{-\frac{x^2}{2}},\quad \lambda\in\mathbb N_0,
$$
we obtain
\begin{align*}
&\partial_{\bsx}^{\bslambda}{\rm e}^{-\bsx^{\rm T}\bsx/2}\bigg|_{\bsx=\Gamma^{-1/2}(\bsy-G_s(\bstheta,\bsxi))}\\
&=\sum_{\boldsymbol m\leq \lfloor\bslambda /2\rfloor}\frac{(-1)^{|\bslambda+\bsm|}\boldsymbol\lambda!}{\boldsymbol m!(\boldsymbol\lambda-2\boldsymbol m)!2^{|\bsm|}}\big(\Gamma^{-1/2}(\bsy-G_s(\bstheta,\bsxi))\big)^{\bslambda-2\bsm}{\rm e}^{-\frac{1}{2}\|\bsy-G_s(\bstheta,\bsxi)\|_{\Gamma^{-1}}^2},
\end{align*}
and hence
\begin{align}
\begin{split}
&\partial_{\bsy}^{\bseta}\partial_{\bsx}^{\bslambda}{\rm e}^{-\bsx^{\rm T}\bsx/2}\bigg|_{\bsx=\Gamma^{-1/2}(\bsy-G_s(\bstheta,\bsxi))}\\
&=\!\sum_{\boldsymbol m\leq \lfloor\bslambda /2\rfloor}\!\!\frac{(-1)^{|\bslambda|+|\bsm|}\boldsymbol\lambda!}{\boldsymbol m!(\boldsymbol\lambda\!-\!2\boldsymbol m)!2^{|\bsm|}}\partial_{\bsy}^{\bseta}\bigg[\!\big(\Gamma^{-1/2}(\bsy\!-\!G_s(\bstheta,\bsxi))\big)^{\bslambda-2\bsm}{\rm e}^{-\frac{1}{2}\|\bsy-G_s(\bstheta,\bsxi)\|_{\Gamma^{-1}}^2}\!\bigg]\!.
\end{split}\label{eq:returntothis2}
\end{align}
The Leibniz product rule implies that
\begin{align}
\begin{split}
&\partial_{\bsy}^{\bseta}\bigg[\big(\Gamma^{-1/2}(\bsy-G_s(\bstheta,\bsxi))\big)^{\bslambda-2\bsm}{\rm e}^{-\frac{1}{2}\|\bsy-G_s(\bstheta,\bsxi)\|_{\Gamma^{-1}}^2}\bigg]\\
&=\sum_{\boldsymbol w\leq \boldsymbol\eta}\binom{\boldsymbol\eta}{\boldsymbol w}\partial_{\bsy}^{\boldsymbol w}\big(\Gamma^{-1/2}(\bsy-G_s(\bstheta,\bsxi))\big)^{\bslambda-2\bsm} \partial_{\bsy}^{\boldsymbol \eta-\boldsymbol w}{\rm e}^{-\frac{1}{2}\|\bsy-G_s(\bstheta,\bsxi)\|_{\Gamma^{-1}}^2}\bigg].
\end{split}\label{eq:returntothis}
\end{align}
We need to estimate the parametric regularity of $\partial_{\bsy}^{\boldsymbol w}\big(\Gamma^{-1/2}(\bsy-G_s(\bstheta,\bsxi))\big)^{\bslambda-2\bsm}$ and $\partial_{\bsy}^{\boldsymbol \eta-\boldsymbol w}{\rm e}^{-\frac{1}{2}\|\bsy-G_s(\bstheta,\bsxi)\|_{\Gamma^{-1}}^2}$.
\begin{lemma}\label{lemma:auxreg1} Let $\boldsymbol w\in\mathbb N_0^k\setminus\{\mathbf 0\}$ and $\bslambda\in\mathbb N_0^k$. Then under assumptions {\rm \ref{eq:A2}}--{\rm \ref{eq:A3}}, there holds
$$
|\partial_{\bsy}^{\boldsymbol w}\big(\Gamma^{-1/2}(\bsy-G_s(\bstheta,\bsxi))\big)^{\bslambda}|\leq \frac{\boldsymbol \lambda!}{(\bslambda-\boldsymbol w)!}\|\bsy-G_s(\bstheta,\bsxi)\|^{|\bslambda|-|\boldsymbol w|}\bigg(\frac{1}{\sqrt{\mu_{\min}}}\bigg)^{|\bslambda|}
$$
for all $\bstheta\in\Theta_s$, $\bsy\in [-K,K]^k$, and $\bsxi\in \Xi$.
\end{lemma}
\proof Let $\boldsymbol w\in \mathbb N_0^k\setminus\{\mathbf 0\}$. We can use Fa\`a di Bruno's formula~\cite{savits} to write
$$
\partial_{\bsy}^{\boldsymbol w}\!\big(\Gamma^{-1/2}(\bsy-G_s(\bstheta,\bsxi))\big)^{\bslambda}\!=\!\sum_{\substack{1\leq |\boldsymbol\ell|\leq |\boldsymbol w|\\ \boldsymbol \ell\in\mathbb N_0^k}}\!\frac{\bslambda!}{(\bslambda-\boldsymbol\ell)!}\big(\Gamma^{-1/2}(\bsy-G_s(\bstheta,\bsxi))\big)^{\boldsymbol\lambda-\boldsymbol\ell}\rho_{\boldsymbol w,\boldsymbol\ell}(\bsy),
$$
where the sequence $(\rho_{\boldsymbol w,\boldsymbol\ell})_{\boldsymbol w\in\mathbb N_0^k,\boldsymbol\ell\in\mathbb Z^k}=(\rho_{\boldsymbol w,\boldsymbol\ell}(\bsy))_{\boldsymbol w\in\mathbb N_0^k,\boldsymbol\ell\in\mathbb Z^k}$ is defined by the recurrence
\begin{align*}
&\rho_{\boldsymbol w,\mathbf 0}\equiv \delta_{\boldsymbol w,\mathbf 0},\\
&\rho_{\boldsymbol w,\boldsymbol\ell}\equiv 0\quad\text{if}~|\boldsymbol w|<|\boldsymbol\ell|~\text{or}~\boldsymbol\ell\not\geq \mathbf 0,~\text{and}\\
&\rho_{\boldsymbol w+\boldsymbol e_j,\boldsymbol\ell}=\sum_{q\in{\rm supp}(\boldsymbol\ell)}\sum_{\mathbf 0\leq \bsm\leq \boldsymbol w}\binom{\boldsymbol w}{\boldsymbol m}\partial_{\bsy}^{\bsm+\boldsymbol e_j}[\Gamma^{-1/2}(\bsy-G_s(\bstheta,\bsxi))]_{q}\rho_{\boldsymbol w-\boldsymbol m,\boldsymbol\ell-\boldsymbol e_{q}}
\end{align*}
otherwise. Similarly to the proof of Lemma~\ref{lemma:lah}, this implies that
\begin{align}
|\partial_{\bsy}^{\boldsymbol w}\big(\Gamma^{-1/2}(\bsy-G_s(\bstheta,\bsxi))\big)^{\bslambda}|\leq\! \sum_{\substack{1\leq |\boldsymbol\ell|\leq |\boldsymbol w|\\ \boldsymbol \ell\in\mathbb N_0^k}}\!\frac{\bslambda!}{(\bslambda-\boldsymbol\ell)!}|\Gamma^{-1/2}(\bsy-G_s(\bstheta,\bsxi))|^{\boldsymbol\lambda-\boldsymbol\ell}\tau_{\boldsymbol w,\boldsymbol\ell},\label{eq:plugback}
\end{align}
where the sequence $(\tau_{\boldsymbol w,\boldsymbol\ell})_{\boldsymbol w\in\mathbb N_0^k,\boldsymbol\ell\in\mathbb Z^k}$ is defined by the recurrence
\begin{align}
&\tau_{\boldsymbol w,\mathbf 0}= \delta_{\boldsymbol w,\mathbf 0},\notag\\
&\tau_{\boldsymbol w,\boldsymbol\ell}= 0\quad\text{if}~|\boldsymbol w|<|\boldsymbol\ell|~\text{or}~\boldsymbol\ell\not\geq \mathbf 0,\quad\text{and}\notag\\
&\tau_{\boldsymbol w+\boldsymbol e_j,\boldsymbol\ell}=\mu_{\min}^{-1/2}\sum_{q\in{\rm supp}(\boldsymbol\ell)}\sum_{\mathbf 0\leq \bsm\leq \boldsymbol w}\binom{\boldsymbol w}{\boldsymbol m}\partial_{\bsy}^{\bsm+\boldsymbol e_j}\![\bsy-G_s(\bstheta,\bsxi)]_{q}\tau_{\boldsymbol w-\boldsymbol m,\boldsymbol\ell-\boldsymbol e_{q}}\label{eq:presimplify}
\end{align}
otherwise. Equation~\eqref{eq:presimplify} simplifies to
$$
\tau_{\boldsymbol w+\boldsymbol e_j,\boldsymbol \ell}= \mu_{\min}^{-1/2}\tau_{\boldsymbol w-\boldsymbol e_j,\boldsymbol \ell},
$$ 
which implies that this sequence has the analytical solution
$$
\tau_{\boldsymbol w,\boldsymbol\ell}=\mu_{\min}^{-|\boldsymbol w|/2}\delta_{\boldsymbol w,\boldsymbol\ell}\quad\text{for all}~\boldsymbol w\in\mathbb N_0^k~\text{and}~\boldsymbol\ell\in \mathbb Z^k.
$$
Plugging this into~\eqref{eq:plugback} yields
\begin{align*}
|\partial_{\bsy}^{\boldsymbol w}\big(\Gamma^{-1/2}(\bsy-G_s(\bstheta,\bsxi))\big)^{\bslambda}|\leq \mu_{\min}^{-|\boldsymbol w|/2}\frac{\boldsymbol \lambda!}{(\boldsymbol\lambda-\boldsymbol w)!}|\Gamma^{-1/2}(\bsy-G_s(\bstheta,\bsxi))|^{\bslambda-\boldsymbol w}.
\end{align*}
The claim follows by applying the inequality
$$
|[\Gamma^{-1/2}(\bsy-G_s(\bstheta,\bsxi))]_\ell|\leq \|\Gamma^{-1/2}(\bsy-G_s(\bstheta,\bsxi))\|\leq \mu_{\min}^{-1/2}\|\bsy-G_s(\bstheta,\bsxi)\|
$$
for $\ell=1,\ldots,k$.\qed
\endproof

\begin{lemma}\label{lemma:auxreg2} Let $\bsnu\in\mathbb N_0^k\setminus\{\mathbf 0\}$.  Then under assumptions {\rm \ref{eq:A2}}--{\rm \ref{eq:A3}}, there holds
$$
|\partial_{\boldsymbol{y}}^{\boldsymbol{\nu}}{\rm e}^{-\frac{1}{2} \|\boldsymbol y-G_s(\boldsymbol\theta,\boldsymbol\xi)\|_{\Gamma^{-1}}^2}|\leq 1.1^kk\mu_{\min}^{-1/2}
$$
for all $\bstheta\in\Theta_s$, $\bsy\in [-K,K]^k$, and $\bsxi\in\Xi$.
\end{lemma}
\proof 
Using Fa\`a di Bruno's formula~\cite{savits}, we obtain
$$
\partial_{\boldsymbol{y}}^{\boldsymbol{\nu}}{\rm e}^{-\frac{1}{2} \|\boldsymbol y-G_s(\boldsymbol\theta,\boldsymbol\xi)\|_{\Gamma^{-1}}^2}=\sum_{1\leq |\boldsymbol\lambda|\leq |\boldsymbol\nu|}\partial_{\boldsymbol x}^{\boldsymbol\lambda}{\rm e}^{-\boldsymbol x^{\rm T}\boldsymbol x/2}\bigg|_{\boldsymbol x=\Gamma^{-1/2}(\boldsymbol y-G_s(\boldsymbol\theta,\boldsymbol\xi))}\rho_{\boldsymbol\nu,\boldsymbol\lambda}(\bsy),
$$
where the coefficient sequence $\rho_{\boldsymbol{\nu},\boldsymbol\lambda}(\bsy)$ can be bounded by
\begin{align*}
&\tau_{\boldsymbol\nu,\mathbf 0}=\delta_{\boldsymbol\nu,\mathbf 0},\\
&\tau_{\boldsymbol\nu,\boldsymbol \lambda}=0\quad\text{if}~|\boldsymbol \nu|<|\boldsymbol \lambda|~\text{or}~\boldsymbol\lambda\not\geq \mathbf 0,\\
&\tau_{\boldsymbol\nu+\boldsymbol e_j,\boldsymbol \lambda}=\sum_{\ell\in{\rm supp}(\boldsymbol\lambda)}\mu_{\min}^{-1/2}\tau_{\boldsymbol \nu,\boldsymbol \lambda-\boldsymbol e_{\ell}}\quad\text{otherwise}.
\end{align*}
It is not difficult to see that
$$
\tau_{\boldsymbol{\nu},\boldsymbol\lambda}=\begin{cases}1&\text{if}~\boldsymbol{\nu}=\boldsymbol\lambda=\mathbf 0,\\
\mu_{\min}^{-1/2}&\text{if}~\boldsymbol{\nu}\neq \mathbf 0~\text{and}~\boldsymbol\lambda=\boldsymbol e_k,~k\geq 1,\\
0&\text{otherwise},\end{cases}
$$
which yields the assertion.\qed
	
By applying Lemmata~\ref{lemma:auxreg1}~and~\ref{lemma:auxreg2} to~\eqref{eq:returntothis2}--\eqref{eq:returntothis}, we obtain that
\begin{align}
&\bigg|\partial_{\bsy}^{\bseta}\partial_{\bsx}^{\bslambda}{\rm e}^{-\bsx^{\rm T}\bsx/2}\bigg|_{\bsx=\Gamma^{-1/2}(\bsy-G_s(\bstheta,\bsxi))}\bigg|\notag\\
&\leq 1.1^kk\mu_{\min}^{-1/2}\sum_{\boldsymbol m\leq \lfloor\bslambda /2\rfloor}\frac{\boldsymbol\lambda!}{\boldsymbol m!(\boldsymbol\lambda-2\boldsymbol m)!2^{|\bsm|}}\notag \\
&\quad \times \sum_{\boldsymbol w\leq\bseta}\binom{\bseta}{\boldsymbol w}\frac{(\bslambda-2\boldsymbol m)!}{(\bslambda-2\boldsymbol m-\boldsymbol w)!}\|\bsy-G_s(\bstheta,\bsxi)\|^{|\bslambda|-2|\boldsymbol m|-|\boldsymbol w|}\bigg(\frac{1}{\sqrt{\mu_{\min}}}\bigg)^{|\bslambda|-2|\boldsymbol m|}\notag\\
&\leq 1.1^kk\bslambda!\bseta!\mu_{\min}^{-(|\bslambda|+1)/2}R^{|\bslambda|}\sum_{\boldsymbol m\leq \lfloor\bslambda /2\rfloor}\frac{1}{\boldsymbol m!} \sum_{\boldsymbol w\leq\boldsymbol\eta}\frac{1}{\boldsymbol w!}\notag\\
&\leq 1.1^kk{\rm e}^{2k}\bslambda!\bseta!\mu_{\min}^{-(|\bslambda|+1)/2}R^{|\bslambda|},\label{eq:premixedbound}
\end{align}
with a constant $R\geq 1$. Thus, the mixed partial derivatives can be bounded uniformly (with a constant depending on the data dimension $k$).

\begin{lemma}\label{lemma:mixedregbound} Let $\bsnu\in\mathbb N_0^s\setminus\{\mathbf 0\}$ and $\bseta\in\mathbb N_0^k\setminus\{\mathbf 0\}$. Then under assumptions~{\rm \ref{eq:A2}}--{\rm \ref{eq:A3}}, there holds
$$
|\partial_{\bsy}^{\bseta}\partial_{\bstheta}^{\bsnu}{\rm e}^{-\frac12 \|\bsy-G_s(\bstheta,\bsxi)\|_{\Gamma^{-1}}^2}|\leq \rev{1.1^k\cdot 2^{\beta (k-1)}}k{\rm e}^{2k}\bseta!(\rev{4^{\beta}}C)^{|\bsnu|}\mu_{\min}^{-|\bsnu|-1/2}R^{|\bsnu|}\rev{(|\bsnu|!)^{\beta}}\bsb^{\bsnu}
$$
for all $\bstheta\in\Theta_s$, $\bsy\in[-K,K]^k$, and $\bsxi\in\Xi$.
\end{lemma}
\proof By plugging the bound~\eqref{eq:premixedbound} into the expression~\eqref{eq:mixedderivatives}, using the sequence $(\rev{\chi}_{\bsnu,\bslambda})$ defined by~\eqref{eq:faa1}--\eqref{eq:faa3} to bound the sequence $(\alpha_{\bsnu,\bslambda}(\bstheta))$, and applying the upper bound~\eqref{eq:lah}, we obtain
\begin{align*}
&|\partial_{\bsy}^{\bseta}\partial_{\bstheta}^{\bsnu}{\rm e}^{-\frac12 \|\bsy-G_s(\bstheta,\bsxi)\|_{\Gamma^{-1}}^2}|\\
&\leq \!1.1^kk{\rm e}^{2k}\bseta!C^{|\bsnu|}\mu_{\min}^{-|\bsnu|-1/2}R^{|\bsnu|}\rev{(|\bsnu|!(|\bsnu|-1)!)^{\beta}}\bsb^{\bsnu}\!\!\!\!\sum_{\substack{\bslambda\in\mathbb N_0^k\\ 1\leq |\bslambda|\leq |\bsnu|}}\!\!\!\!\frac{1}{\rev{((|\bsnu|\!-\!|\bslambda|)!(|\bslambda|\!-\!1)!)^{\beta}}},
\end{align*}
where
\begin{align*}
&\sum_{\substack{\bslambda\in\mathbb N_0^k\\ 1\leq |\bslambda|\leq |\bsnu|}}\frac{1}{\rev{((|\bsnu|-|\bslambda|)!(|\bslambda|-1)!)^{\beta}}}\\&=\sum_{\ell=1}^{|\bsnu|}\frac{1}{\rev{((|\bsnu|-\ell)!(\ell-1)!)^{\beta}}}\sum_{\substack{\bslambda\in\mathbb N_0^k\\ |\bslambda|=\ell}}1\\
&=\sum_{\ell=1}^{|\bsnu|}\frac{1}{\rev{((|\bsnu|-\ell)!(\ell-1)!)^{\beta}}}\binom{k+\ell-1}{\ell}\\
&\rev{\,\leq\sum_{\ell=1}^{|\bsnu|}\frac{1}{\rev{((|\bsnu|-\ell)!(\ell-1)!)^{\beta}}}\binom{k+\ell-1}{\ell}^{\beta}}\\
&=\frac{1}{\rev{((|\bsnu|-1)!)^{\beta}}}\sum_{\ell=1}^{|\bsnu|}\rev{\binom{|\bsnu|-1}{\ell-1}^{\beta}\binom{k+\ell-1}{\ell}^{\beta}}\\
&\rev{\,\leq\frac{1}{\rev{((|\bsnu|-1)!)^{\beta}}}\bigg(\sum_{\ell=1}^{|\bsnu|}\rev{\binom{|\bsnu|-1}{\ell-1}\binom{k+\ell-1}{\ell}\bigg)^{\beta}}}\\
&\leq \frac{1}{\rev{((|\bsnu|-1)!)^{\beta}}}\rev{\bigg(\sum_{\ell=1}^{|\bsnu|}2^{|\bsnu|+k+\ell-2}\bigg)^{\beta}}\\
&\leq \rev{\frac1{2^\beta} \cdot 4^{\beta|\bsnu|}\cdot 2^{\beta k}}\frac{1}{\rev{((|\bsnu|-1)!)^{\beta}}},
\end{align*}
\rev{where the third to last inequality follows from Jensen's inequality~\eqref{eq:jensenineq}. This proves the assertion.}\qed\endproof
\subsection{Parametric regularity of the outer integral}\label{sec:outer}
Ultimately we will be interested in applying QMC to approximate the outer integral in
\begin{align}
\int_{[-K,K]^k} \!\!\log\bigg(\!C_{k,\Gamma}\!\!\int_{{\Theta_s}}\! {\rm e}^{-\frac{1}{2} \|\boldsymbol y-G_s\!(\boldsymbol\theta,\boldsymbol\xi)\|_{\Gamma^{-1}}^2}\,{\rm d}\bstheta\bigg)\!\bigg(\!C_{k,\Gamma}\!\!\int_{\Theta_s} \!{\rm e}^{-\frac{1}{2} \|\boldsymbol y-G_s\!({\boldsymbol\theta},\boldsymbol\xi)\|_{\Gamma^{-1}}^2}{\rm d}{\boldsymbol\theta}\bigg){\rm d}\boldsymbol y.\label{eq:outerintdef}
\end{align}
In this section, we will estimate the parametric regularity of the integrand with respect to $\boldsymbol y$. Let us first investigate the logarithmic term.

Let $\boldsymbol\nu\in\mathbb N_0^k\setminus\{\mathbf 0\}$. We may again use Fa\`a di Bruno's formula~\cite{savits} to obtain
\begin{align*}
&\partial_{\boldsymbol y}^{\boldsymbol\nu}\log\bigg(C_{k,\Gamma}\int_{\Theta_s} {\rm e}^{-\frac{1}{2} \|\boldsymbol y-G_s(\boldsymbol\theta,\boldsymbol\xi)\|_{\Gamma^{-1}}^2}\,{\rm d}\bstheta\bigg)\\
&=\sum_{\lambda=1}^{|\boldsymbol{\nu}|}(-1)^{\lambda+1}(\lambda-1)!\bigg(\int_{\Theta_s} {\rm e}^{-\frac{1}{2} \|\boldsymbol y-G_s(\boldsymbol\theta,\boldsymbol\xi)\|_{\Gamma^{-1}}^2}\,{\rm d}\bstheta\bigg)^{-\lambda}\alpha_{\boldsymbol\nu,\lambda}(\boldsymbol y),
\end{align*}
where
\begin{align*}
&\alpha_{\boldsymbol\nu,0}=\delta_{\boldsymbol\nu,\mathbf 0},\\
&\alpha_{\boldsymbol\nu,\lambda}=0\quad\text{if}~|\boldsymbol \nu|<\lambda,~\text{and}\\
&\alpha_{\boldsymbol\nu+\boldsymbol e_j,\lambda}(\bsy)=\sum_{\mathbf 0\leq \boldsymbol m\leq\boldsymbol\nu}\binom{\boldsymbol\nu}{\boldsymbol m}\partial_{\boldsymbol y}^{\boldsymbol m+\boldsymbol e_j}\bigg(\int_{\Theta_s} {\rm e}^{-\frac{1}{2} \|\boldsymbol y-G_s(\boldsymbol\theta,\boldsymbol\xi)\|_{\Gamma^{-1}}^2}\,{\rm d}\bstheta\bigg)\alpha_{\boldsymbol \nu-\boldsymbol m,\lambda-1}(\bsy)
\end{align*}
otherwise. Similarly to Lemma~\ref{lemma:auxreg2},  we can bound the sequence $\alpha_{\boldsymbol{\nu},\lambda}(\bsy)$ uniformly by
\begin{align*}
&\rev{\chi}_{\boldsymbol{\nu},0}=\delta_{\boldsymbol{\nu},\mathbf 0},\\
&\rev{\chi}_{\boldsymbol{\nu},\lambda}=0\quad\text{if}~|\boldsymbol{\nu}|<\lambda,\quad\text{and}\\
&\rev{\chi}_{\boldsymbol{\nu}+\boldsymbol e_j,\lambda}=\sum_{\mathbf 0\leq \boldsymbol m\leq \boldsymbol{\nu}}\binom{\boldsymbol{\nu}}{\boldsymbol m}1.1^kk\mu_{\min}^{-1/2}\rev{\chi}_{\boldsymbol{\nu}-\boldsymbol m,\lambda-1}\quad\text{otherwise}.
\end{align*}
Comparing this recursion with the characteristic recursion of the \emph{Stirling numbers of the second kind} reveals that
$$
\rev{\chi}_{\boldsymbol{\nu},\lambda}\leq 1.1^{k|\boldsymbol{\nu}|}k^{|\boldsymbol{\nu}|}\mu_{\min}^{-|\boldsymbol{\nu}|/2}S(|\boldsymbol{\nu}|,\lambda)
$$
and altogether we obtain that
\begin{align*}
&\bigg|\partial_{\boldsymbol y}^{\boldsymbol\nu}\log\bigg(C_{k,\Gamma}\int_{\Theta_s}{\rm e}^{-\frac{1}{2} \|\boldsymbol y-G_s(\boldsymbol\theta,\boldsymbol\xi)\|_{\Gamma^{-1}}^2}\,{\rm d}\bstheta\bigg)\bigg|    \\
&\leq 1.1^{k|\boldsymbol{\nu}|}k^{|\boldsymbol{\nu}|}\mu_{\min}^{-|\boldsymbol{\nu}|/2}\sum_{\lambda=1}^{|\boldsymbol{\nu}|}(\lambda-1)!\bigg(\int_{\Theta_s} {\rm e}^{-\frac{1}{2} \|\boldsymbol y-G_s(\boldsymbol\theta,\boldsymbol\xi)\|_{\Gamma^{-1}}^2}\,{\rm d}\bstheta\bigg)^{-\lambda}S(|\boldsymbol{\nu}|,\lambda).
\end{align*}
It is well-known that (cf., e.g., \cite[Lemma~A.3]{beck12})
$$
\sum_{\lambda=1}^{|\boldsymbol{\nu}|}(\lambda-1)!S(|\boldsymbol{\nu}|,\lambda)\leq \frac{|\bsnu|!}{(\log 2)^{|\bsnu|}}\quad\text{for all}~\bsnu\in\mathbb N_0^k.
$$
Jensen's inequality implies that
\begin{align}
\notag\bigg(\!\int_{\Theta_s}\! {\rm e}^{-\!\frac{1}{2} \|\boldsymbol y-G_s(\boldsymbol\theta,\boldsymbol\xi)\|_{\Gamma^{-1}}^2}\,{\rm d}\bstheta\!\bigg)^{-1}&\leq \int_{\Theta_s}\! {\rm e}^{\frac{1}{2} \|\boldsymbol y-G_s(\boldsymbol\theta,\boldsymbol\xi)\|_{\Gamma^{-1}}^2}\,{\rm d}\bstheta\\
&\leq  {\rm e}^{\frac{1}{2\mu_{\min}}(kK^2+2\sqrt{k}KC+C^2)}\label{eq:jensen}
\end{align}
for all $\bsy\in[-K,K]^k$, so we obtain
\begin{align}
&\partial_{\boldsymbol y}^{\boldsymbol\nu}\log\bigg(C_{k,\Gamma}\int_{\Theta_s}{\rm e}^{-\frac{1}{2} \|\boldsymbol y-G_s(\boldsymbol\theta,\boldsymbol\xi)\|_{\Gamma^{-1}}^2}\,{\rm d}\bstheta\bigg)\notag\\
&\leq 1.1^{k|\boldsymbol{\nu}|}k^{|\boldsymbol{\nu}|}\mu_{\min}^{-|\boldsymbol{\nu}|/2}{\rm e}^{\frac{|\bsnu|}{2\mu_{\min}}(kK^2+2\sqrt{k}KC+C^2)}\frac{|\bsnu|!}{(\log 2)^{|\bsnu|}}\notag%
\end{align}
for all $\bsnu\in\mathbb N_0^k$ and $\bsy\in[-K,K]^k$.

\begin{lemma}\label{reg:outer} Let $\bsnu\in\mathbb N_0^k\setminus\{\mathbf 0\}$. Then under assumptions~{\rm \ref{eq:A2}}--{\rm \ref{eq:A3}}, there holds
\begin{align*}
&\partial_{\bsy}^{\bsnu}\bigg(\bigg(C_{k,\Gamma}\int_{\Theta_s} {\rm e}^{-\frac{1}{2}\|\bsy-G_s(\bstheta,\bsxi)\|_{\Gamma^{-1}}^2}\,{\rm d}\bstheta\bigg)\log\bigg(C_{k,\Gamma}\int_{\Theta_s} {\rm e}^{-\frac{1}{2}\|\bsy-G_s(\bstheta,\bsxi)\|_{\Gamma^{-1}}^2}\,{\rm d}\bstheta\bigg)\bigg)\\
&\leq 1.1^{k(1+|\bsnu|)}k^{1+|\bsnu|}\mu_{\min}^{-(|\bsnu|+1)/2}\frac{|\bsnu|!}{(\log 2)^{|\bsnu|}}{\rm e}^{1+\frac{|\bsnu|}{2\mu_{\min}}(kK^2+2\sqrt k KC+C^2)}
\end{align*}
for all $\bstheta\in\Theta_s$, $\bsy\in [-K,K]^k$, and $\bsxi\in\Xi$.
\end{lemma}
\proof By Leibniz product rule, we obtain
\begin{align*}
&\partial_{\bsy}^{\bsnu}\bigg(\bigg(C_{k,\Gamma}\int_{\Theta_s} {\rm e}^{-\frac{1}{2}\|\bsy-G_s(\bstheta,\bsxi)\|_{\Gamma^{-1}}^2}\,{\rm d}\bstheta\bigg)\log\bigg(C_{k,\Gamma}\int_{\Theta_s} {\rm e}^{-\frac{1}{2}\|\bsy-G_s(\bstheta,\bsxi)\|_{\Gamma^{-1}}^2}\,{\rm d}\bstheta\bigg)\bigg)\\
&=\sum_{\boldsymbol m\leq \bsnu}\binom{\bsnu}{\boldsymbol m}\partial_\bsy^{\bsnu-\boldsymbol m}\bigg(C_{k,\Gamma}\int_{\Theta_s} {\rm e}^{-\frac{1}{2}\|\bsy-G_s(\bstheta,\bsxi)\|_{\Gamma^{-1}}^2}\,{\rm d}\bstheta\bigg)\\
&\quad\quad \times \partial_\bsy^{\boldsymbol m}\log\bigg(C_{k,\Gamma}\int_{\Theta_s} {\rm e}^{-\frac{1}{2}\|\bsy-G_s(\bstheta,\bsxi)\|_{\Gamma^{-1}}^2}\,{\rm d}\bstheta\bigg)\\
&\leq 1.1^kk\mu_{\min}^{-1/2}\sum_{\boldsymbol m\leq \bsnu}\binom{\bsnu}{\boldsymbol m}1.1^{k|\boldsymbol m|}k^{|\boldsymbol m|}\mu_{\min}^{-|\boldsymbol m|/2}{\rm e}^{\frac{|\boldsymbol m|}{2\mu_{\min}}(kK^2+2\sqrt k KC+C^2)}\frac{|\boldsymbol m|!}{(\log 2)^{|\boldsymbol m|}}\\
&\leq 1.1^{k(1\!+|\bsnu|)}k^{1\!+|\bsnu|}\mu_{\min}^{-(|\bsnu|\!+1)/2}\frac{1}{(\log 2)^{|\bsnu|}}{\rm e}^{\frac{|\bsnu|}{2\mu_{\min}}(kK^2+2\sqrt k KC+C^2)}\sum_{\ell=0}^{|\bsnu|}\ell!\sum_{\substack{|\boldsymbol m|=\ell\\ \boldsymbol m\leq\bsnu}}\binom{\bsnu}{\boldsymbol m}\\
&\leq 1.1^{k(1\!+|\bsnu|)}k^{1\!+|\bsnu|}\mu_{\min}^{-(|\bsnu|\!+1)/2}\frac{1}{(\log 2)^{|\bsnu|}}{\rm e}^{\frac{|\bsnu|}{2\mu_{\min}}(kK^2+2\sqrt k KC+C^2)}\underset{\leq {\rm e}\cdot (|\bsnu|!)}{\underbrace{\sum_{\ell=0}^{|\bsnu|}\ell!\binom{|\bsnu|}{\ell}}},
\end{align*}
as desired.\qed\endproof
\section{QMC error for the single integrals}\label{sec:single}
Let us first consider the problem of approximating the single integral~\eqref{eq:innerintdef}, which we denote by
$$
\int_{\Theta_s}f(\bstheta,\bsy)\,{\rm d}\bstheta,\quad f(\bstheta,\bsy):=C_{k,\Gamma}{\rm e}^{-\frac12 \|\bsy-G_s(\bstheta,\bsxi)\|_{\Gamma^{-1}}^2},
$$
for $\bstheta\in\Theta_s$ with fixed $\bsy\in [-K,K]^k$ and $\bsxi\in\Xi$, 
by designing a randomly shifted rank-1 lattice rule % 
$$
\overline{Q}_{n,R}(f)=\frac{1}{nR}\sum_{r=1}^R \sum_{i=1}^n f(\{\bst_i+\boldsymbol\Delta_r\}-\tfrac{\mathbf 1}{\mathbf 2}),
$$
where $\bst_i=\{i\boldsymbol z/n\}$ are lattice points for $i\in\{1,\ldots,n\}$ corresponding to some generating vector $\boldsymbol z\in\mathbb N^s$ and $\boldsymbol\Delta_1,\ldots,\boldsymbol\Delta_R\overset{\rm i.i.d.}{\sim}\mathcal U([0,1]^s)$. We can use Lemma~\ref{lemma:innerreg} together with standard QMC theory (cf., e.g., \cite{kuonuyenssurvey}) to obtain the following result.
\begin{theorem}\label{thm1}
Let $n=2^m$, $m\geq 0$. Then under assumptions {\rm \ref{eq:A2}}--{\rm \ref{eq:A3}}, it is possible to use the CBC algorithm to obtain a generating vector $\boldsymbol z\in\mathbb N^s$ such that the randomly shifted rank-1 lattice rule for the integrand $f$ of~\eqref{eq:innerintdef} satisfies the root-mean-square error estimate
$$
\sqrt{\mathbb E_{\boldsymbol\Delta}|I_s(f)-\overline{Q}_{n,R}(f)|^2}\leq Cn^{\max\{-1/p+1/2,-1+\delta\}},
$$
where the constant $C>0$ is independent of the dimension $s$, provided that the product-and-order dependent (POD) weights
$$
\gamma_{\setu}:=\bigg(\rev{(|\setu|!)^{\beta}}\prod_{j\in\setu}\frac{\rev{c}_j}{\sqrt{\frac{2\zeta(2\lambda)}{(2\pi^2)^\lambda}}}\bigg)^{\frac{2}{1+\lambda}},~\lambda:=\begin{cases}\frac{p}{2-p}&\text{if}~p\in\rev{(\frac23,\rev{\frac1{\beta})}},\\ \frac{1}{2-2\delta}&\text{if}~p\in\rev{(0,\min\{\frac23,\frac1{\beta}\}}],~\rev{p\neq\frac{1}{\beta}}\end{cases}
$$
are used as inputs to the CBC algorithm. Here, \rev{$\delta\in(0,\frac12)$} is arbitrary and we define
$$
\rev{c}_j:=\frac{\rev{2^\beta} C}{\sqrt{\mu_{\min}}}b_j,\quad j\in\{1,\ldots,s\}.
$$
\end{theorem}

In addition, we wish to approximate the integral~\eqref{eq:outerintdef}, which we denote by
\begin{align}\label{IkK}
I_{k,K}(g):=\int_{[-K,K]^k}g(\bsy)\,{\rm d}\bsy,\!\!\quad g(\bsy)=\bigg(\int_{\Theta_s}f(\bstheta,\bsy)\,{\rm d}\bstheta\bigg)\log\bigg(\int_{\Theta_s}f(\bstheta,\bsy)\,{\rm d}\bstheta\bigg),
\end{align}
for $\bsy\in [-K,K]^k$ and fixed $\bsxi\in\Xi$ using another randomly shifted rank-1 lattice rule
$$
\overline{Q}_{n,R,K}(f)=\frac{1}{nR}\sum_{r=1}^R \sum_{i=1}^n g(2K\{\bst_i+\boldsymbol\Delta_r\}-\mathbf 1 K),
$$
where the lattice points $\bst_i=\{i\bsz/n\}$ for $i\in\{1,\ldots,n\}$ corresponding to some generating vector $\bsz\in\mathbb N^k$ and $\boldsymbol\Delta_1,\ldots,\boldsymbol\Delta_R\overset{\rm i.i.d.}{\sim}\mathcal U([0,1]^k)$ have been scaled to the computational domain $[-K,K]^k$. For this integral, we cannot in general expect the QMC convergence to be independent of the dimensionality of the data $k$. Thus the best we can hope for is to try to minimize the constant of the QMC error estimate. This can again be achieved using standard QMC theory (cf., e.g., \cite{kuonuyenssurvey}).

\begin{theorem}\label{thm2}
Let $n=2^m$, $m\geq 0$. Then under assumptions~{\rm \ref{eq:A2}}--{\rm \ref{eq:A3}}, it is possible to use the CBC algorithm to obtain a generating vector $\bsz\in\mathbb N^k$ such that the randomly shifted rank-1 lattice rule applied to the outermost integral $I_{k,K}(g)$ of~\eqref{IkK} satisfies the root-mean-square error estimate
$$
\sqrt{\mathbb E_{\boldsymbol\Delta}|I_{k,K}(g)-\overline{Q}_{n,R,K}(g)|^2}\leq Cn^{-1+\delta},
$$
where $g$ denotes the corresponding integrand in~\eqref{IkK} and the constant $C>0$ is bounded provided that the order-dependent weights
$$
\gamma_{\setu}:=\bigg(|\setu|!\prod_{j\in\setu}\frac{1.1^kk\mu_{\min}^{-1/2}{\rm e}^{\frac{1}{2\mu_{\min}}(kK^2+2\sqrt k KC+C^2)}}{\log(2)\sqrt{2\zeta(2\lambda)/(2\pi^2)^{\lambda}}}\bigg)^{\frac{2}{1+\lambda}},\quad\lambda:=\frac{1}{2-2\delta},
$$
are used as inputs to the CBC algorithm, where $\delta>0$ is arbitrary. %
\end{theorem}
\rev{
\begin{remark}\label{rem:scale}\rm
For any prescribed tolerance $\varepsilon>0$, the truncation parameter $K$ can be chosen sufficiently large such that the cutoff error is bounded by $\varepsilon$. The subsequent change of variables from $[-K,K]^k$ to $[0,1]^k$ scales the derivative bounds by constants depending on $K$. In consequence, once $K=K(\varepsilon)$ has been fixed, the number of QMC points can be selected in order to achieve the desired accuracy for the truncated integral. The corresponding constants in the QMC error bounds depend on $K$, so that larger values of $K$ will require a larger number of QMC points as can be seen from Theorem \ref{thm2}.
\end{remark}}
\begin{remark}\rm
For clarity, we have omitted the dependence on the number of random shifts $R$ in the convergence rates of Theorems~\ref{thm1} and~\ref{thm2} as well as in the subsequent convergence analysis in Sections~\ref{sec:ftp} and~\ref{sec:stp}.
\end{remark}

\section{Full tensor product cubature for the double integral}\label{sec:ftp}
The presented regularity analysis allows the use of QMC integration for the inner and outer integral. A straightforward combination of both approximations, i.e., the inner and the outer integral, leads to the so-called full tensor grid approach. Recall that
$$
{\rm EIG}=\int_{\rev{\mathbb R^k}}\int_{{\Theta_s}}\log\bigg(\frac{\pi(\bstheta|\bsy,\bsxi)}{\pi(\bstheta)}\bigg)\pi(\bstheta|\bsy,\bsxi)\pi(\bsy|\bsxi)\,{\rm d}\bstheta\,{\rm d}\bsy
$$
for a given design $\bsxi\in\Xi$. The expected information gain can be equivalently formulated as
\begin{align*}
{\rm EIG}&=\log C_{k,\Gamma}-\frac{k}{2}\\
&\quad -\!\int_{\rev{\mathbb R^k}} \!\log\!\bigg(\!\int_{{\Theta_s}}\!\!\!C_{k,\Gamma}{\rm e}^{-\frac{1}{2} \|\boldsymbol y-G_s(\boldsymbol\theta,\boldsymbol\xi)\|_{\Gamma^{-1}}^2}{\rm d}\boldsymbol\theta\bigg)\!\!\int_{{\Theta_s}}\!\!\!C_{k,\Gamma}{\rm e}^{-\frac{1}{2} \|\boldsymbol y-G_s(\boldsymbol\theta,\boldsymbol\xi)\|_{\Gamma^{-1}}^2}{\rm d}\boldsymbol\theta{\rm d}\boldsymbol y
\end{align*}
i.e., the goal of the computation is the double integral
$$
\int_{\mathbb R^k}\log\bigg(\int_{\Theta_s} C_{k,{\Gamma}}{\rm e}^{-\frac{1}{2}\|\bsy-G_s(\bstheta,\bsxi)\|_{\Gamma^{-1}}^2}\,{\rm d}\bstheta\bigg)\int_{\Theta_s} C_{k,{\Gamma}}{\rm e}^{-\frac{1}{2}\|\bsy-G_s(\bstheta,\bsxi)\|_{\Gamma^{-1}}^2}\,{\rm d}\bstheta\,{\rm d}\bsy\,
$$ and after truncation of the integration domain
\begin{align*}
\mathcal I_K &=\int_{[-K,K]^k}\log\bigg(\int_{\Theta_s} C_{k,{\Gamma}}{\rm e}^{-\frac{1}{2}\|\bsy-G_s(\bstheta,\bsxi)\|_{\Gamma^{-1}}^2}\,{\rm d}\bstheta\bigg)\\
&\quad\quad\times \int_{\Theta_s} C_{k,{\Gamma}}{\rm e}^{-\frac{1}{2}\|\bsy-G_s(\bstheta,\bsxi)\|_{\Gamma^{-1}}^2}\,{\rm d}\bstheta\,{\rm d}\bsy.
\end{align*}

In this section, we let $g(x)=x\log x$, $U_m=[0,1]^m$, $\bsgamma=(\gamma_{\setu})_{\setu\subseteq\{1:m\}}$ is a sequence of positive weights, and denote by \cs{$\mathcal W_{m,\bsgamma}$} the unanchored, weighted Sobolev space of absolutely continuous functions $F\!:U_m\to\mathbb R$ with square-integrable first order mixed partial derivatives, equipped with the norm
$$
\|F\|_{\mathcal W_{m,\bsgamma}}^2=\sum_{\setu\subseteq\{1:m\}}\frac{1}{\gamma_{\setu}}\int_{\rev{U_{|\setu|}}}\bigg(\int_{\rev{U_{m-|\setu|}}}\partial_{\setu}F(\bsy)\,{\rm d}\bsy_{-\setu}\bigg)^2\,{\rm d}\bsy_{\setu}.
$$
In what follows, we shall focus on the cubature approximation of% 
 \begin{align}
\mathcal I:=\int_{U_k} g\bigg(\int_{U_s} f(\bstheta,\bsy)\,{\rm d}\bstheta\bigg)\,{\rm d}\bsy,\notag%
\end{align}
with the understanding that we can recover the integral $\mathcal I_K$ corresponding to~\eqref{eq:integralofinterest} by an affine change of variables.

Defining a sequence of QMC cubature operators by
\begin{align}
I^{(1)}F:=\!\int_{U_k}\!F(\bsy)\,{\rm d}\bsy\approx 2^{-\ell}\!\sum_{i=1}^{2^{\ell}}F(\bsy_i^{(\ell)})=:\!Q_\ell^{(1)} F,\quad \ell=0,1,2,\ldots,\label{eq:l01}
\end{align}
where $\bsy_i^{(\ell)}=\{i\boldsymbol z/2^{\ell}+\boldsymbol\Delta\}$ denote $k$-dimensional lattice points~\eqref{eq:latticepoint} for $i\in\{1,\ldots,2^{\ell}\}$ and $\ell\geq 0$ with a single random shift $\boldsymbol\Delta\sim\mathcal U([0,1]^k)$ for a given function $F\in \mathcal W_{k,\boldsymbol\gamma_1}$ and
\begin{align}\label{eq:l02}
I^{(2)}F:=\!\int_{{U_s}}\!F(\bstheta)\,{\rm d}\bstheta\approx 2^{-\ell}\!\sum_{i=1}^{2^{\ell}}F(\bstheta_i^{(\ell)})=:\!Q_\ell^{(2)} F,\quad \ell=0,1,2,\ldots,
\end{align}
where $\bstheta_i^{(\ell)}$ similarly denote $s$-dimensional lattice points for $i\in\{1,\ldots,2^{\ell}\}$  for $\ell\geq 0$ with a single random shift for a given function $F\in \mathcal W_{s,\boldsymbol\gamma_2}$, where $\boldsymbol\gamma_1=(\gamma_{1,\setu})_{\setu\subseteq\{1:k\}}$ and $\boldsymbol\gamma_2=(\gamma_{2,\setu})_{\setu\subseteq\{1:s\}}$ are sequences of positive weights, we approximate the integral $\mathcal I$ by
$$
\mathcal I \approx {Q_\ell^{(1)}} g(Q_\ell^{(2)}f)\,.
$$
To this end, we assume that the QMC cubature operators satisfy the error bounds
\begin{align}\label{eq:approx1}
\| I^{(1)}F-Q_\ell^{(1)} F\|_{\boldsymbol{\Delta}}:=\sqrt{\mathbb{E}_{\cs{{\boldsymbol{\Delta}}}}[|I^{(1)}F-Q_\ell^{(1)} F|^2]}\leq C_12^{-\delta_1\ell}\|F\|_{\mathcal W_{k,\boldsymbol\gamma_1}}\quad
\end{align}
for all $F\in \mathcal W_{k,\boldsymbol\gamma_1}$
and
\begin{align}\label{eq:approx222}
\| I^{(2)}F- Q_\ell^{(2)} F\|_{\cs{{\boldsymbol{\Delta}}}}:=\sqrt{\mathbb{E}_{\boldsymbol{\Delta}}[|I^{(2)}F- Q_\ell^{(2)} F|^2]}\leq C_22^{-\delta_2\ell}\|F\|_{\mathcal W_{s,\boldsymbol\gamma_2}}\quad
\end{align}
for all $F\in\mathcal W_{s,\boldsymbol\gamma_2}$, with $C_1,C_2,\delta_1,\delta_2>0$ and $\ell\geq 0$. Furthermore, we assume that the integrand satisfies the following assumptions.

\begin{assumption}\label{assumptionone}\rm We assume that $f\!:U_s\times U_k\to\mathbb R$ is a continuous function which satisfies the following:
\begin{enumerate}[label=(A2.\arabic*),align=right,leftmargin=2.5\parindent]%
\item $\sup_{\bsy\in U_k}\|f(\cdot,\bsy)\|_{\cs{\mathcal W_{s,\bsgamma_2}}}\leq C_{k,\bsgamma_1}$ for some constant $C_{k,\bsgamma_1}>0$ independently of $\bsy\in U_{k}$; \label{eq:A21}
\item $g(I^{(2)}f)\in \cs{\mathcal W_{k,\bsgamma_1}}$ and $g(Q_{\ell}^{(2)}f)\in \cs{\mathcal W_{k,\bsgamma_1}}$ for all $\ell\geq 0$.\label{eq:A22}
\end{enumerate}
\end{assumption}

Then we have the following result.
\begin{theorem}
Suppose that $f\!:U_s\times U_k\to\mathbb R$ is a continuous function satisfying assumptions {\rm \ref{eq:A21}}--{\rm \ref{eq:A22}} and $g(x)=x\log x$. Then there holds for the cubature rules satisfying~\eqref{eq:approx1}--\eqref{eq:approx222} that%
   \begin{align*}%
&\| \mathcal I -  {Q_{\ell_1}^{(1)}} g(Q_{\ell_2}^{(2)}f)\|_{\cs{{\boldsymbol{\Delta}}}}\\
&\quad\le C_12^{-\delta_1\ell_1}\|g(I^{(2)}f)\|_{\mathcal W_{k,\boldsymbol\gamma_1}}+ \tilde C_22^{-\alpha \delta_2{\ell_2}}\sup_{\bsy\in U_k}\|f(\cdot, \bsy)\|_{\mathcal W_{s,\boldsymbol\gamma_2}}
\end{align*}
for $0<\alpha<1$. In particular, when using the same number of cubature points for the inner and outer integral, i.e., $\ell_1=\ell_2=\ell$, the rate of convergence is given by
\begin{align*}%
&\| \mathcal I -  {Q_{\ell}^{(1)}} g(Q_{\ell}^{(2)}f)\|_{\cs{{\boldsymbol{\Delta}}}}\le C\  2^{-\frac{\min\{\delta_1, \alpha \delta_2\}}{2} \ell }\,.
\end{align*}
\end{theorem}
\begin{proof}
Using the H\"older continuity of the mapping $g(x)=x\log x$, we obtain
\begin{align*}
&\| \mathcal I -  {Q_{\ell_1}^{(1)}} g(Q_{\ell_2}^{(2)}f)\|_{\cs{{\boldsymbol{\Delta}}}}\\
&\quad\le \| \mathcal I -  {Q_{\ell_1}^{(1)}} g(I^{(2)}f)\|_{\cs{{\boldsymbol{\Delta}}}} +\| {Q_{\ell_1}^{(1)}} g(I^{(2)}f) -  {Q_{\ell_1}^{(1)}} g(Q_{\ell_2}^{(2)}f)\|_{\cs{{\boldsymbol{\Delta}}}}\\
&\quad\le C_12^{-\delta_1\ell_1}\|g(I^{(2)}f)\|_{\mathcal W_{k,\boldsymbol\gamma_1}} + \sup_{\bsy\in U_k}  \| g(I^{(2)}f(\cdot, \bsy)) -  g(Q_{\ell_2}^{(2)}f(\cdot, \bsy))\|_{\cs{{\boldsymbol{\Delta}}}}\\
&\quad\le C_12^{-\delta_1\ell_1}\|g(I^{(2)}f)\|_{\mathcal W_{k,\boldsymbol\gamma_1}} + \sup_{\bsy\in U_k}  \| I^{(2)}f(\cdot, \bsy) -  Q_{\ell_2}^{(2)}f(\cdot, \bsy)\|^\alpha_{\cs{{\boldsymbol{\Delta}}}}\\
&\quad\le C_12^{-\delta_1\ell_1}\|g(I^{(2)}f)\|_{\mathcal W_{k,\boldsymbol\gamma_1}}+ \tilde C_22^{-\alpha \delta_2{\ell_2}}\sup_{\bsy\in U_k}\|f(\cdot, \bsy)\|_{\mathcal W_{s,\boldsymbol\gamma_2}}
\end{align*}
for $0<\alpha<1$.
The rate of convergence for the case $\ell_1=\ell_2=\ell$ directly follows from the general bound.\qed
\end{proof}
\begin{theorem}\label{lemma:ftprate}
Under assumptions {\rm \ref{eq:A2}}--{\rm \ref{eq:A3}}, with $p\in(0,\rev{\min\{\frac23,\frac1{\beta}\}}]$, $\rev{p\neq\frac1{\beta}}$, in~{\rm \ref{eq:A1}}, and $\ell_1=\ell_2=\ell$, the rate of convergence for full tensor product approximation of the double integral $\mathcal I_K$ satisfies
\begin{align*}%
&\textup{R.M.S.~error}\le C\  (2^\ell)^{-\frac12+\delta},
\end{align*}
where $C>0$ is a constant, $\delta>0$ is arbitrary, and the cubature point set of the outer cubature operator~\eqref{eq:l01} is scaled to the cube $[-K,K]^k$.
\end{theorem}
\begin{proof}
The conditions~\ref{eq:A21}--\ref{eq:A22} for the model problem follow from the analysis carried out in Subsections~\ref{sec:inner} and~\ref{sec:outer}.\qed\end{proof}

Note that the rate of convergence is halved due to the nested integral in the design problem.
    
We also obtain the following result as a corollary.
\begin{corollary}
If the randomly shifted lattice rule~\eqref{eq:l02} for the inner integral is obtained using the CBC algorithm with the input weights described in Theorem~\ref{thm1}, then the constant $C>0$ in Theorem~\ref{lemma:ftprate} is independent of the dimension~$s$.
\end{corollary}

\section{Sparse tensor product cubature for the double integral}\label{sec:stp}
Forming a direct composition of the QMC cubatures corresponding to both the inner and outer integral leads to a suboptimal cubature convergence rate as shown in the previous section. To this end, inspired by~\cite{griebel}, we construct instead a sparse tensor product cubature based on two families of QMC rules.

Let $Q_{\ell}^{(1)}$ and $Q_{\ell}^{(2)}$ be cubature rules satisfying~\eqref{eq:approx1}--\eqref{eq:approx222}. For the outer integral, we consider the following %
difference operators
$$
\Delta_\ell^{(1)} F:=\begin{cases}Q_{\ell}^{(1)} F-Q_{\ell-1}^{(1)}F&\text{if}~\ell \geq 1,\\ Q_{0}^{(1)}F&\text{if}~\ell=0\end{cases}
$$
for a function $F\in \mathcal W_{k,\boldsymbol\gamma_1}$. 
The triangle inequality and the approximation assumption on the cubature operators \eqref{eq:approx1} lead to a bound of the difference operators
\begin{align*}
\|\Delta_\ell^{(1)}\|_{\cs{\mathcal W_{k,\boldsymbol\gamma_1}}\to \mathbb R}&=\sup_{\substack{f\in \mathcal W_{k,\boldsymbol\gamma_1}\\ \|f\|_{\mathcal W_{k,\boldsymbol\gamma_1}}\le 1}} \| \Delta_\ell^{(1)} (f)\|_{\cs{{\boldsymbol{\Delta}}}}  \\
&\le  \sup_{\substack{f\in \mathcal W_{k,\boldsymbol\gamma_1}\\ \|f\|_{\mathcal W_{k,\boldsymbol\gamma_1}}\le 1}} \| I^{(1)}(f) - Q_\ell^{(1)}(f)\|_{\cs{{\boldsymbol{\Delta}}}} + \| I^{(1)}(f) - Q_{\ell-1}^{(1)}(f)\|_{\cs{{\boldsymbol{\Delta}}}}\\
&\le  \tilde C_12^{-\delta_1\ell}.%
\end{align*}

We further define the generalized difference operators for $g(x)=x\log x$ by setting
$$
 \Delta_{\ell}^{(2)}(F):=\begin{cases}g( Q_{\ell}^{(2)} F)-g ( Q_{\ell-1}^{(2)}F)&\text{if}~\ell\geq 1,\\ g( Q_0^{(2)}F)&\text{if}~\ell=0\end{cases}
$$
for a slightly more general sequence of cubature operators
\begin{align}\label{eq:shiftdef}
Q_{\ell}^{(2)}(F):=2^{-\ell-\ell_0^{(2)}}\sum_{i=1}^{2^{\ell+\ell_0^{(2)}}}F(\bstheta_i^{(\ell+\ell_0^{(2)})}),\quad \ell=0,1,2,\ldots,
\end{align}
i.e., they correspond to~\eqref{eq:shiftdef} for a fixed offset $\ell_0^{(2)}\in\mathbb N_0$, $F\in\mathcal W_{s,\bsgamma_2}$. We have introduced the offset parameter above in order to balance the error contributions stemming from the nonlinear term $g$ in the upcoming convergence analysis.

For $\sigma>0$ and $f\in C(U_s\times U_k)$, we define the generalized sparse grid cubature operator
$$
\mathcal Q_{L,\sigma}(f)=\sum_{\sigma\ell_1+\frac{\ell_2}{\sigma}\leq L}\Delta_{\ell_1}^{(1)} \Delta_{\ell_2}^{(2)}(f).
$$
This operator can alternatively be written as
\begin{align*}
\mathcal Q_{L,\sigma}(f)&=\sum_{\sigma\ell_1+\frac{\ell_2}{\sigma}\leq L}\Delta_{\ell_1}^{(1)}\Delta_{\ell_2}^{(2)}(f)=\sum_{\ell_1=0}^{L/\sigma}\sum_{\ell_2=0}^{\sigma L-\sigma^2\ell_1}\Delta_{\ell_1}^{(1)}\Delta_{\ell_2}^{(2)}(f)\\
&=\sum_{\ell_1=0}^{L/\sigma}\Delta_{\ell_1}^{(1)}\sum_{\ell_2=0}^{\sigma L-\sigma^2\ell_1}\Delta_{\ell_2}^{(2)}(f)=\sum_{\ell_1=0}^{L/\sigma}\Delta_{\ell_1}^{(1)}g( Q_{\sigma L-\sigma^2\ell_1}^{(2)}f).
\end{align*}

In addition to Assumptions~\ref{eq:A21}--\ref{eq:A22}, we will make the following assumptions about the integrand.
\begin{assumption}\label{assumptiontwo}\rm We assume that $f\!:U_s\times U_k\to\mathbb R$ is a continuous function which satisfies the following:
\begin{enumerate}[label=(A2.\arabic*),start=3,align=right,leftmargin=2.5\parindent]%
\item $\sup_{\|h\|_{H^{k}([0,1]^{k})}\leq 1}\|\langle h,f\rangle_{H^{k}([0,1]^{k})}\|_{\cs{\mathcal W_{s,\bsgamma_2}}}\leq C<\infty$ for some constant $C>0$ independently of $s$;\label{eq:A23}
\item There exists a constant $c>0$ such that $f(\bstheta,\bsy)\geq c$ for all $\bstheta\in U_s$ and $\bsy\in U_k$. \label{eq:A24}
\end{enumerate}
\end{assumption}

We obtain the following result.
\begin{theorem}\label{thm:stpwodimind}
 Suppose that $f\!:U_s\times U_k\to\mathbb R$ is a continuous function satisfying assumptions {\rm \ref{eq:A21}}--{\rm \ref{eq:A24}} and $g(x)=x\log x$. Then there exist constants $C_1, C_2, C_3 \ge 0$ such that the approximation error of the generalized sparse grid operator satisfies the bound
 \begin{align*}
\|\mathcal I-\mathcal Q_{L,\sigma}(f)\|_{\cs{{\boldsymbol{\Delta}}}}
&\leq C_12^{-\delta_1L/\sigma}\|g(I^{(2)}f)\|_{\mathcal W_{k,\boldsymbol\gamma_1}}\\
&\quad+C_2 2^{-\delta_2\sigma L} \frac{1-2^{-(\delta_1-\delta_2\sigma^2)(L/\sigma +1))}}{1-2^{-(\delta_1-\delta_2\sigma^2)}}\\
&\quad+C_3 \sum_{\ell_1=0}^{L/\sigma}  2^{-2\delta_2(\sigma L-\sigma^2\ell_1+\ell_0^{(2)})}
\end{align*}
if $\delta_1-\delta_2\sigma^2\neq 0$, and
 \begin{align*}
\|\mathcal I-\mathcal Q_{L,\sigma}(f)\|_{\cs{{\boldsymbol{\Delta}}}}
\leq C_1 2^{-\delta_2\sigma L}(L/\sigma +1)
+C_3 \sum_{\ell_1=0}^{L/\sigma}  2^{-2\delta_2(\sigma L-\sigma^2\ell_1+\ell_0^{(2)})}
\end{align*}
otherwise.
\end{theorem}
\begin{proof}
For each shift $\boldsymbol\Delta$, we can split the approximation error of the generalized sparse grid operator as follows
\begin{align*}
&|\mathcal I-\mathcal Q_{L,\sigma}(f)|\notag\\
&=|\mathcal I-Q_{L/\sigma}^{(1)}g(I^{(2)}f)+Q_{L/\sigma}^{(1)}g(I^{(2)}f)-\mathcal Q_{L,\sigma}(f)|\notag\\
&\leq |(I^{(1)}-Q_{L/\sigma}^{(1)})g(I^{(2)}f)|+\bigg|\mathcal \displaystyle\sum_{\ell_1=0}^{L/\sigma}\Delta_{\ell_1}^{(1)}(g(I^{(2)}f)-g( Q_{\sigma L-\sigma^2\ell_1}^{(2)}f))\bigg|\notag\\
&\leq |(I^{(1)}-Q_{L/\sigma}^{(1)})g(I^{(2)}f)|\notag\\
& \quad +\bigg|\mathcal \displaystyle\sum_{\ell_1=0}^{L/\sigma}\Delta_{\ell_1}^{(1)}(I^{(2)}f- Q_{\sigma L-\sigma^2\ell_1}^{(2)}f) \log(I^{(2)}f)\bigg|\notag\\
& \quad +\bigg|\mathcal \displaystyle\sum_{\ell_1=0}^{L/\sigma}\Delta_{\ell_1}^{(1)}(\log(I^{(2)}f)-\log( Q_{\sigma L-\sigma^2\ell_1}^{(2)}f))Q_{\sigma L-\sigma^2\ell_1}^{(2)}f\bigg|\notag\\
&\leq |(I^{(1)}-Q_{L/\sigma}^{(1)})g(I^{(2)}f)| \\
&\quad+\bigg|\mathcal \displaystyle\sum_{\ell_1=0}^{L/\sigma}\Delta_{\ell_1}^{(1)}(I^{(2)}f- Q_{\sigma L-\sigma^2\ell_1}^{(2)}f) \log(I^{(2)}f)\bigg| \\
& \quad+\bigg|\mathcal \displaystyle\sum_{\ell_1=0}^{L/\sigma}\Delta_{\ell_1}^{(1)}\bigg(\frac{Q_{\sigma L-\sigma^2\ell_1}^{(2)}f- I^{(2)}f}{Q_{\sigma L-\sigma^2\ell_1}^{(2)}f}\bigg)Q_{\sigma L-\sigma^2\ell_1}^{(2)}f\bigg|\\
& \quad+\bigg|\mathcal \displaystyle\sum_{\ell_1=0}^{L/\sigma}\Delta_{\ell_1}^{(1)}\bigg(\log\bigg(\frac{I^{(2)}f}{Q_{\sigma L-\sigma^2\ell_1}^{(2)}f}\bigg)-\bigg(\frac{I^{(2)}f}{Q_{\sigma L-\sigma^2\ell_1}^{(2)}f}-1\bigg)\bigg)Q_{\sigma L-\sigma^2\ell_1}^{(2)}f\bigg|.
\end{align*}
The first term can be bounded by
\begin{align}
\|(I^{(1)}-Q_{L/\sigma}^{(1)})g(I^{(2)}f)\|_{\cs{{\boldsymbol{\Delta}}}}\leq C_12^{-\delta_1L/\sigma}\|g(I^{(2)}f)\|_{\mathcal W_{k,\boldsymbol\gamma_1}}\label{eq:dimind1}
\end{align}
by the approximation property~\eqref{eq:approx1}.

Further, we have for the second term
\begin{align}
&\bigg|\mathcal \displaystyle\sum_{\ell_1=0}^{L/\sigma}\Delta_{\ell_1}^{(1)}(I^{(2)}f- Q_{\sigma L-\sigma^2\ell_1}^{(2)}f) \log(I^{(2)}f)\bigg| \notag \\
&\quad \le \sum_{\ell_1=0}^{L/\sigma} \| \Delta_{\ell_1}^{(1)}\|_{\mathcal W_{k,\boldsymbol\gamma_1} \to \mathbb R} \| (I^{(2)}f- Q_{\sigma L-\sigma^2\ell_1}^{(2)}f) \log(I^{(2)}f)\|_{\mathcal W_{k,\boldsymbol\gamma_1}}.\notag
\end{align}
Here, we use the Leibniz product rule and Cauchy--Schwarz inequality to estimate
\begin{align*}
&\| (I^{(2)}f- Q_{\sigma L-\sigma^2\ell_1}^{(2)}f) \log(I^{(2)}f)\|_{\mathcal W_{k,\boldsymbol\gamma_1}}^2\\
&=\sum_{\setu\subseteq\{1:{k}\}}\frac{1}{\gamma_{1,\setu}} \int_{[0,1]^{|\setu|}}\\
&\!\quad\!\times\bigg(\!\int_{[0,1]^{{k}-|\setu|}}\partial_{\setu}((I^{(2)}f(\cdot,\bsy)-Q_{\sigma L-\sigma^2\ell_1}^{(2)}f(\cdot,\bsy))\log(I^{(2)}f(\cdot,\bsy))\,{\rm d}\bsy_{-\setu}\bigg)^{\!\!2}{\rm d}\bsy_{\setu}\\
&=\sum_{\setu\subseteq\{1:{k}\}}\frac{1}{\gamma_{1,\setu}}\int_{[0,1]^{|\setu|}}\\
&\quad \!\times\! \bigg(\!\int_{[0,1]^{{k}-\!|\setu|}}\!\sum_{\setv\subseteq\setu}\partial_{\setv}(I^{(2)}\!f(\cdot,\bsy)\!-\!Q_{\sigma L\!-\!\sigma^2\ell_1}^{(2)}f(\cdot,\bsy))\partial_{\!-\!\setv}\!\log(I^{(2)}\!f(\cdot,\bsy))\,{\rm d}\bsy_{-\!\setu}\!\bigg)^{\!\!2}{\rm d}\bsy_{\setu}\\
&\leq \sum_{\setu\subseteq\{1:{k}\}}\sum_{\setv\subseteq\setu}\frac{2^{|\setu|}}{\gamma_{1,\setu}}\int_{[0,1]^{|\setu|}}\\
&\quad\! \times\!\bigg(\!\int_{[0,1]^{{k}-\!|\setu|}}\!\partial_{\setv}(I^{(2)}\!f(\cdot,\bsy)\!-\!Q_{\sigma L-\sigma^2\ell_1}^{(2)}f(\cdot,\bsy))\partial_{-\setv}\log(I^{(2)}f(\cdot,\bsy))\,{\rm d}\bsy_{-\setu}\bigg)^2\,{\rm d}\bsy_{\setu}\\
&\leq \sum_{\setu\subseteq\{1:{k}\}}\frac{2^{{k}+|\setu|}}{\gamma_{1,\setu}}\int_{[0,1]^{k}}[\partial_{\setv}(I^{(2)}f(\cdot,\bsy)-Q_{\sigma L-\sigma^2\ell_1}^{(2)}f(\cdot,\bsy))]^2\\
&\quad\quad\quad\quad\quad\quad\quad\quad\quad\quad \!\times\![\partial_{-\setv}\log(I^{(2)}f(\cdot,\bsy))]^2\,{\rm d}\bsy\\
&\leq C_{\bsgamma_1,{k}}^2\|I^{(2)}f-Q_{\sigma L-\sigma^2\ell_1}^{(2)}f\|_{H^{k}([0,1]^{k})}^2,
\end{align*}
where we made use of the inequalities
$$
\sum_{\setu\subseteq\{1:k\}}\sum_{\setv \subseteq\setu}a_{\setv}\leq 2^{k}\sum_{\setu\subseteq\{1:k\}}a_{\setu}\quad\text{for all}~a_{\setu}\geq 0,~\setu\subseteq\{1:k\},
$$
and
$$
\sum_{\setu\subseteq\{1:k\}}\|\partial_{\setu}F\|_{L^2([0,1]^{k})}^2\leq \sum_{|\bsnu|\leq k}\|\partial^{\bsnu}F\|_{L^2([0,1]^{k})}^2=:\|F\|_{H^{k}([0,1]^{k})}^2,
$$
and define
\begin{align}\label{eq:dimind2}
C_{\bsgamma_1,k}:=\frac{2^k}{\sqrt{\min_{\setu\subseteq\{1:k\}}\gamma_{1,\setu}}}\max_{\substack{\bsy\in [0,1]^{k}\\ \setu\subseteq\{1:k\}}}|\partial_{\setu}\log (I^{(2)}f(\cdot,\bsy))|.
\end{align}
Furthermore,
\begin{align}
&\|I^{(2)}f-Q_{\sigma L-\sigma^2\ell_1}^{(2)}f\|_{H^{k}([0,1]^{k})}\notag\\
&=\sup_{\|h\|_{H^{k}([0,1]^{k})}\leq 1}|\langle h,I^{(2)}f-Q_{\sigma L-\sigma^2\ell_1}^{(2)}f\rangle_{H^{k}([0,1]^{k})}|\notag\\
&=\sup_{\|h\|_{H^{k}([0,1]^{k})}\leq 1}|(I^{(2)}-Q_{\sigma L-\sigma^2\ell_1}^{(2)})\langle h,f\rangle_{H^{k}([0,1]^{k})}|\notag\\
&\leq \|I^{(2)}-Q_{\sigma L-\sigma^2\ell_1}^{(2)}\|_{\cs{\mathcal W_{s,\bsgamma_2}}}\sup_{\|h\|_{H^{k}([0,1]^{k})}\leq 1}\|\langle h,f\rangle_{H^{k}([0,1]^{k})}\|_{\cs{\mathcal W_{s,\bsgamma_2}}}\label{eq:productbit}
\end{align}
and we obtain
\begin{align*}
&\bigg\|\mathcal \displaystyle\sum_{\ell_1=0}^{L/\sigma}\Delta_{\ell_1}^{(1)}(I^{(2)}f- Q_{\sigma L-\sigma^2\ell_1}^{(2)}f) \log(I^{(2)}f)\bigg\|_{\cs{{\boldsymbol{\Delta}}}} \\
&\quad \le \sum_{\ell_1=0}^{L/\sigma} \tilde C_2 2^{-\delta_1\ell_1}  2^{-\delta_2(\sigma L-\sigma^2\ell_1)}\notag\\
&\quad= \tilde C_2 2^{-\delta_2\sigma L} \frac{1-2^{-(\delta_1-\delta_2\sigma^2)(L/\sigma +1)}}{1-2^{-(\delta_1-\delta_2\sigma^2)}}\,,
\end{align*}
if $\delta_1-\delta_2\sigma^2\neq 0$, else we obtain the bound $\tilde C_2 2^{-\delta_2\sigma L}(L/\sigma +1)$.
Similarly, we obtain for the third term %
\begin{align*}
&\bigg|\mathcal \displaystyle\sum_{\ell_1=0}^{L/\sigma}\Delta_{\ell_1}^{(1)}\bigg(\frac{Q_{\sigma L-\sigma^2\ell_1}^{(2)}f- I^{(2)}f}{Q_{\sigma L-\sigma^2\ell_1}^{(2)}f}\bigg)Q_{\sigma L-\sigma^2\ell_1}^{(2)}f\bigg| \notag \\
&\quad \le \sum_{\ell_1=0}^{L/\sigma} \| \Delta_{\ell_1}^{(1)}\|_{\mathcal W_{k,\boldsymbol\gamma_1} \to \mathbb R} \| Q_{\sigma L-\sigma^2\ell_1}^{(2)}f- I^{(2)}f \|_{\mathcal W_{k,\boldsymbol\gamma_1}}\notag \\
&\quad=\sum_{\ell_1=0}^{L/\sigma}  \| \Delta_{\ell_1}^{(1)}\|_{\mathcal W_{k,\boldsymbol\gamma_1} \to \mathbb R} \sup_{\substack{h\in (\mathcal W_{k,\boldsymbol\gamma_1})'\\ \|h\|_{(\mathcal W_{k,\boldsymbol\gamma_1})'}\le 1}} |\langle h, I^{(2)}f- Q_{\sigma L-\sigma^2\ell_1}^{(2)}f\rangle_{(\mathcal W_{k,\boldsymbol\gamma_1})', \mathcal W_{k,\boldsymbol\gamma_1}}|\notag \\
&\quad \le \sum_{\ell_1=0}^{L/\sigma} \| \Delta_{\ell_1}^{(1)}\|_{\mathcal W_{k,\boldsymbol\gamma_1} \to \mathbb R} \|I^{(2)}- Q_{\sigma L-\sigma^2\ell_1}^{(2)}\|_{\mathcal W_{s,\boldsymbol\gamma_2}}\notag\\
&\qquad\times \sup_{h\in (\mathcal W_{k,\boldsymbol\gamma_1})', \|h\|_{(\mathcal W_{k,\boldsymbol\gamma_1})'}\le 1} \|\langle h, f\rangle_{(\mathcal W_{k,\boldsymbol\gamma_1})', \mathcal W_{k,\boldsymbol\gamma_1}}\|_{\mathcal W_{s,\boldsymbol\gamma_2}}. \notag 
\end{align*}
In complete analogy to the second term, there holds (with a modified constant $\tilde C_2$) such that
\begin{align*}
&\bigg\|\mathcal \displaystyle\sum_{\ell_1=0}^{L/\sigma}\Delta_{\ell_1}^{(1)}\bigg(\frac{Q_{\sigma L-\sigma^2\ell_1}^{(2)}f- I^{(2)}f}{Q_{\sigma L-\sigma^2\ell_1}^{(2)}f}\bigg)Q_{\sigma L-\sigma^2\ell_1}^{(2)}f\bigg\|_{\cs{{\boldsymbol{\Delta}}}}\\
&\le \sum_{\ell_1=0}^{L/\sigma} \tilde C_2 2^{-\delta_1\ell_1}  2^{-\delta_2(\sigma L-\sigma^2\ell_1)}\notag\\
&\quad= \tilde C_2 2^{-\delta_2\sigma L} \frac{1-2^{-(\delta_1-\delta_2\sigma^2)(L/\sigma +1)}}{1-2^{-(\delta_1-\delta_2\sigma^2)}}\,,
\end{align*}
if $\delta_1-\delta_2\sigma^2\neq 0$, else we obtain the bound $\tilde C_2 2^{-\delta_2\sigma L}(L/\sigma +1)$.
The last term results from the error of considering the linear approximation for the inner operator. However, this error can be made arbitrarily small by adjusting the first level of the inner approximation, i.e., we have
\begin{align*}
&\bigg|\mathcal \displaystyle\sum_{\ell_1=0}^{L/\sigma}\Delta_{\ell_1}^{(1)}\bigg(\log\bigg(\frac{I^{(2)}f}{Q_{\sigma L-\sigma^2\ell_1}^{(2)}f}\bigg)-\bigg(\frac{I^{(2)}f}{Q_{\sigma L-\sigma^2\ell_1}^{(2)}f}-1\bigg)\bigg)Q_{\sigma L-\sigma^2\ell_1}^{(2)}f\bigg| \\
&\le \sum_{\ell_1=0}^{L/\sigma} C \sup_{\bsy\in [0,1]^k} \bigg| \log\bigg(\frac{I^{(2)}f}{Q_{\sigma L-\sigma^2\ell_1}^{(2)}f}\bigg)-\bigg(\frac{I^{(2)}f}{Q_{\sigma L-\sigma^2\ell_1}^{(2)}f}-1\bigg) \bigg|\,.
\end{align*}
By Jensen's inequality,
$$
\frac{1}{Q_{\sigma L-\sigma^2\ell_1}^{(2)}f}\leq Q_{\sigma L-\sigma^2\ell_1}^{(2)}\bigg(\frac{1}{f}\bigg)\leq \frac{1}{c},
$$
which implies that
\begin{align*}
&\bigg\|\frac{I^{(2)}f}{Q_{\sigma L-\sigma^2\ell_1}^{(2)}f}-1\bigg\|_{\cs{{\boldsymbol{\Delta}}}}\le C 2^{-\delta_2(\sigma L-\sigma^2\ell_1+\ell_0^{(2)})}.
\end{align*}
So far, our analysis has been independent of the choice of the offset parameter $\ell_0^{(2)}$ in~\eqref{eq:shiftdef} since the error rate is not affected. However, here we estimate
\begin{align*}
&\bigg\|\mathcal \displaystyle\sum_{\ell_1=0}^{L/\sigma}\Delta_{\ell_1}^{(1)}\bigg(\log\bigg(\frac{I^{(2)}f}{Q_{\sigma L-\sigma^2\ell_1}^{(2)}f}\bigg)-\bigg(\frac{I^{(2)}f}{Q_{\sigma L-\sigma^2\ell_1}^{(2)}f}-1\bigg)\bigg)Q_{\sigma L-\sigma^2\ell_1}^{(2)}f\bigg\|_{\cs{{\boldsymbol{\Delta}}}} \\
&\le C_3 \sum_{\ell_1=0}^{L/\sigma}  2^{-\delta_2(\sigma L-\sigma^2\ell_1+\ell_0^{(2)})} \,
\end{align*}
with $C_3>0$, and choose the offset $\ell_0^{(2)}$ to be large enough to balance the contribution of this term with the other terms appearing in the overall error bound.%
\qed  
\end{proof}
The proof technique relies on the fact that the nonlinearity resulting from the logarithm can be bounded, i.e., the lower level approximation is already good enough. The error analysis is therefore tailored for the specific optimal design setting. We expect that similar strategies based on linearization can be applied to more general settings and will be subject to future work. Note that the analysis from \cite{griebel} does not give convergence in the current setting, since convergence rates of the QMC method are not available for the logarithm of the inner integral in the corresponding norm. Furthermore, the above strategy could be applied to other types of cubature operators.  
\rev{We keep the analysis for general $\sigma>0$ since this parameter encodes the effective regularity of the integrand and directly influences the attainable QMC and sparse grid convergence rates. Allowing for $\sigma\neq1$ makes it possible, at least in principle, to balance regularity assumptions against the number of quadrature points and hence to control the trade-off between different approximation errors; this idea was first explored in~\cite{griebel14}. While this flexibility is not exploited in the present numerical experiments, which focus on the standard choice $\sigma=1$ for clarity, the general setting is of independent theoretical interest and provides a basis for future adaptive strategies.}

\begin{theorem}\label{lemma:genericstp}
Under assumptions {\rm \ref{eq:A2}}--{\rm \ref{eq:A3}}, with $p\in(0,\rev{\min\{\frac23,\frac{1}{\beta}}\}]$, $\rev{p\neq\frac{1}{\beta}}$, in~{\rm \ref{eq:A1}}, and $\sigma= 1$, the rate of convergence for the sparse tensor product approximation of the double integral $\mathcal I_K$ satisfies
 \begin{align*}
\textup{R.M.S.~error}
\leq C (2^L)^{- 1+\delta}(L +1)
\end{align*}
for $\delta>0$ arbitrary and an appropriately chosen lower level  $\ell_0^{(2)}$ in~\eqref{eq:shiftdef}, where the cubature point set of the outer cubature operator~\eqref{eq:l01} is scaled to the cube $[-K,K]^k$.
\end{theorem}

\begin{proof}
The conditions~\ref{eq:A21}--\ref{eq:A24} are satisfied by the parametric analysis carried out in Section~\ref{sec:reg}. Especially, conditions~\ref{eq:A21}--\ref{eq:A22} follow from the analysis in Subsections~\ref{sec:inner} and~\ref{sec:outer}, \ref{eq:A23} follows from the mixed regularity analysis in Subsection~\ref{sec:mixed}, and \ref{eq:A24} follows similarly to~\eqref{eq:jensen}.\qed
\end{proof}

Similarly to the case of the full tensor cubature, we also obtain the following as a corollary.

\begin{corollary}
If the randomly shifted lattice rule~\eqref{eq:l02} for the inner integral is obtained using the CBC algorithm with the input weights given by
\begin{align}
\gamma_{2,\setu}=\bigg(\rev{(|\setu|!)^{\beta}}\prod_{j\in\setu}\frac{\rev{c}_j}{\sqrt{2\zeta(2\lambda)/(2\pi^2)^\lambda}}\bigg)^{\frac{2}{\lambda+1}},\quad \lambda=\frac{p}{2-p},\label{eq:newweightsX}
\end{align}
where $\rev{c}_j=4^\beta\rev{C}R\mu_{\min}^{-1}b_j$  and  $\delta\in(0,1/2)$ is arbitrary, then the constant $C>0$ in Theorem~\ref{lemma:genericstp} is independent of the dimension $s$.
\end{corollary}
\begin{proof}
The choice of weights~\eqref{eq:newweightsX} ensures that the term $\|g(I^{(2)}f)\|_{\mathcal W_{k,\bsgamma_1}}$ in~\eqref{eq:dimind1} and the constant $C_{\bsgamma_1,k}$ in~\eqref{eq:dimind2} can be bounded independently of $s$. %

In~\eqref{eq:productbit}, we can estimate
\begin{align*}
&\sup_{\|h\|_{H^k([0,1]^k)}\leq 1}\|\langle h,f\rangle_{H^k([0,1]^k)}\|_{\cs{\mathcal W_{s,\bsgamma_2}}}^2\\
&\leq \sum_{\setu\subseteq\{1:s\}}\frac{1}{\gamma_{2,\setu}}\int_{[0,1]^s}|\langle h,\partial_{\setu}f\rangle_{H^k([0,1]^k)}|^2\,{\rm d}\bstheta\\
&\leq \sum_{\setu\subseteq\{1:s\}}\frac{1}{\gamma_{2,\setu}}\int_{[0,1]^s}\|\partial_{\setu}f\|_{H^k([0,1]^k)}^2\,{\rm d}\bstheta\\
&\leq \!\rev{1.1^{2k}\!\cdot\! 2^{2\beta (k-1)}k^2{\rm e}^{4k}\!\bigg(\!\sum_{\substack{\bseta\in\mathbb N_0^k\\ |\bseta|\leq k}}(\bseta!)^2\!\bigg)\!\!\sum_{\setu\subseteq\{1:s\}}\!\!\frac{1}{\gamma_{2,\setu}}(4^\beta C)^{2|\setu|}\mu_{\min}^{-2|\setu|-1}R^{2|\setu|}(|\setu|!)^{2\beta}\bsb_{\setu}^2}
\end{align*}
by Lemma~\ref{lemma:mixedregbound}, and it is a consequence of standard QMC theory~\cite{kuonuyenssurvey} that the choice of weights~\eqref{eq:newweightsX} results in the dimension independence of the constant $\tilde C_2>0$ in the proof of Theorem~\ref{thm:stpwodimind}. The dimension independence of the remaining constants follows from this.\qed
\end{proof}

\section{Numerical experiments}\label{sec:numex}
Let $D=(0,1)^2$. We consider the elliptic PDE
\begin{align}
\begin{cases}
-\nabla \cdot (a(\bsx,\bstheta)\nabla u(\bsx,\bstheta))=10x_1,&\bsx\in D,~\bstheta\in [-1/2,1/2]^{100},\\
u(\cdot,\bstheta)|_{\partial D}=0,&\bstheta\in [-1/2,1/2]^{100},
\end{cases}\label{eq:numexpde}
\end{align}
equipped with the parametric PDE coefficients
\begin{itemize}[align=right,leftmargin=1.2\parindent]
\item[(i)]\! $a(\bsx,\bstheta)\!=\!1+0.1\!\sum_{j=1}^{100} j^{-2}\theta_j\sin(\pi j x_1)\sin(\pi j x_2)$, $\bstheta\in [-1/2,1/2]^{100}$;
\item[(ii)]\! $a(\bsx,\bstheta)\!=\!1+\frac{0.1}{\sqrt 6}\!\sum_{j=1}^{100} j^{-2}\!\sin(2\pi \theta_j)\sin(\pi j x_1)\sin(\pi j x_2)$, $\bstheta\in [-1/2,1/2]^{100}$.
\end{itemize}
It is a consequence of standard elliptic regularity theory that the variational solution corresponding to the problem~\eqref{eq:numexpde} satisfies $u(\cdot,\bstheta)\in H^2(D)\cap H_0^1(D)$ for all $\bstheta\in[-1/2,1/2]^{100}$. Especially, there exists a solution to the variational formulation of the PDE which is continuous with respect to the spatial variable $\bsx\in D$ for all $\bstheta\in[-1/2,1/2]^{100}$ by the standard Sobolev embedding---meaning that point evaluation is a bounded operation. Assumption~\ref{eq:A1} has been verified, e.g., in \cite{cohen10}, and Assumptions~\ref{eq:A2} and~\ref{eq:A3} are trivially fulfilled.

The goal is to find a design $\bsxi^*$ from the set $$\Xi=\{(\bsx_1,\bsx_2,\bsx_3)\in\Upsilon^3\mid \bsx_i\ne \bsx_j~\text{for}~i\neq j\},$$where
\begin{align*}
\Upsilon=&\{(0.25,0.25),(0.25,0.50),(0.25,0.75),\\
&~(0.50,0.25),(0.50,0.50),(0.50,0.75),\\
&~(0.75,0.25),(0.75,0.50),(0.75,0.75)\},
\end{align*}
maximizing the expected information gain~\eqref{eq:numexref} subject to the observation operator
$$
G_{100}(\bstheta,\bsxi)=(u(\bsx,\bstheta))_{\bsx\in \bsxi},\quad \bstheta\in[-1/2,1/2]^{100},~\bsxi\in\Xi.
$$

First, we investigate the numerical approximation of the high-dimensional double integral~\eqref{eq:integralofinterest} appearing in the expression for the EIG~\eqref{eq:numexref}. To this end, we set \rev{$K=1/2$, i.e. the truncated integration domain is $[-1/2,1/2]^k$} with $k=3$ and $\Gamma=0.01I_k$ for the estimated noise level. \rev{The value of the cutoff $K=1/2$ chosen here is justified by the fact that the range of the forward mapping is approximately $[0.38,0.41]$. Using a larger value of the cutoff would lead to longer preasymptotic regimes since most of the cubature nodes would ``miss'' the region containing most of the mass of the integrand. In a practical implementation, one could use, e.g., Laplace approximation to detect where the mass of the posterior distribution lies.}

\rev{We }use the following approximation schemes:
\begin{itemize}[align=right,leftmargin=1.5\parindent]%
\item[(a)] Full tensor product (FTP) cubature: we take the composition of two randomly shifted rank-1 lattice rules $Q_\ell^{(1)}$ and $Q_\ell^{(2)}$ consisting of $n=2^{\ell+1}$ cubature nodes for $\ell=0,1,2,\ldots$. The expected convergence rate in this case is essentially $\mathcal O(N^{-1/2})$, where $N=n^2$ is the total number of integrand evaluations.
\item[(b)] Sparse tensor product (STP) cubature: we use Smolyak's construction to form a cubature rule for the double integral, viz.
$$
\mathcal Q_L=\sum_{\ell_1+\ell_2\leq L}\Delta_{\ell_1}^{(1)}\Delta_{\ell_2}^{(2)},
$$
where $g(x)=x\log x$ and the difference cubature operators are defined by
\begin{align*}
&\Delta_{\ell}^{(1)}F:=\begin{cases}Q_{\ell}^{(1)}F-Q_{\ell-1}^{(1)}F&\text{if}~\ell>0,\\ Q_0^{(1)}F&\text{if}~\ell=0,\end{cases}\\
&\Delta_{\ell}^{(2)}F:=\begin{cases}g(Q_{\ell}^{(2)}F)-g(Q_{\ell-1}^{(2)}F)&\text{if}~\ell>0,\\ g(Q_0^{(2)}F)&\text{if}~\ell=0.\end{cases}
\end{align*} 
Here, $Q_{\ell}^{(1)}$ and $Q_{\ell}^{(2)}$ denote randomly shifted rank-1 lattice rules with $n=2^{\ell+1}$ cubature nodes for $\ell=0,1,2,\ldots$. The expected convergence rate in this case is essentially $\mathcal O(N^{-1}\log N)$ for problem (i), where $N$ is the total number of integrand evaluations.
\item[(c)] In order to extract a theoretically advantageous rate for the periodic parameterization together with the STP construction, we repeat experiments (a) and (b) for the periodically parameterized input random field by replacing the cubatures $Q_{\ell}^{(2)}$ corresponding to the outer integrals over ${\rev{\mathbb R^k}}$ with a $k$-dimensional Smolyak cubature rule
$$
Q_{\ell-2}^{(2)}=\sum_{\substack{\max\{0,\ell-k+1\}\leq |\bsalpha|\leq \ell\\ \bsalpha\in\mathbb N_0^k}}(-1)^{\ell-|\bsalpha|}\binom{k-1}{\ell-|\bsalpha|}\bigotimes_{j=1}^k \mathcal{U}_{\alpha_j},\quad \ell=2,3,\ldots,
$$
where $\mathcal{U}_{m}$ is a univariate trapezoidal rule with $n=\begin{cases}1&\text{if}~m=0\\2^m+1&\text{otherwise}\end{cases}$ nodes. Note that we have shifted the indexing of the outer cubature rules by 2 in order to balance the number of function evaluations with the inner integral for larger $n$.
\end{itemize}

\begin{figure}[!t]
\subfloat{
\begin{tikzpicture}
\node (img) {\includegraphics[height=.365\textwidth,trim=.5cm .5cm 0cm 0cm,clip]{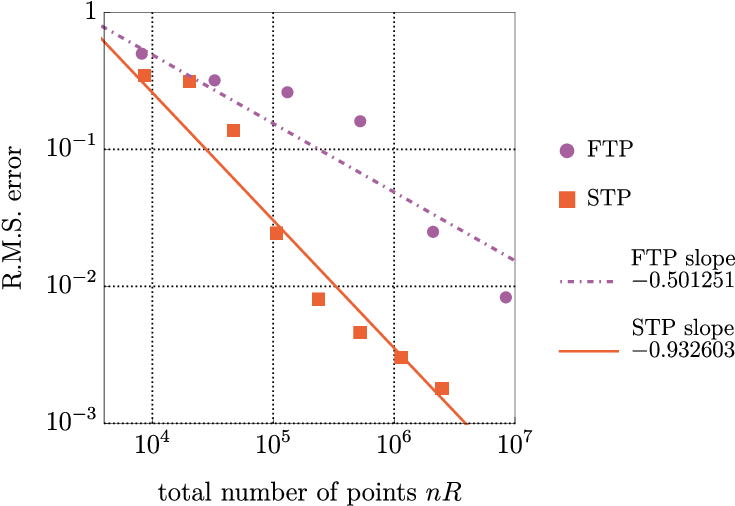}};
\node[below=of img,node distance=0cm,xshift=-.55cm,yshift=1.2cm]{\!\!\!\!\!total number of points $NR$};
\node[left=of img,node distance=0cm,rotate=90,anchor=center,xshift=.25cm,yshift=-1.0cm]{R.M.S. error};
\end{tikzpicture}}
\includegraphics[height=.428\textwidth,trim=0cm 0cm 0cm -0.85cm,clip]{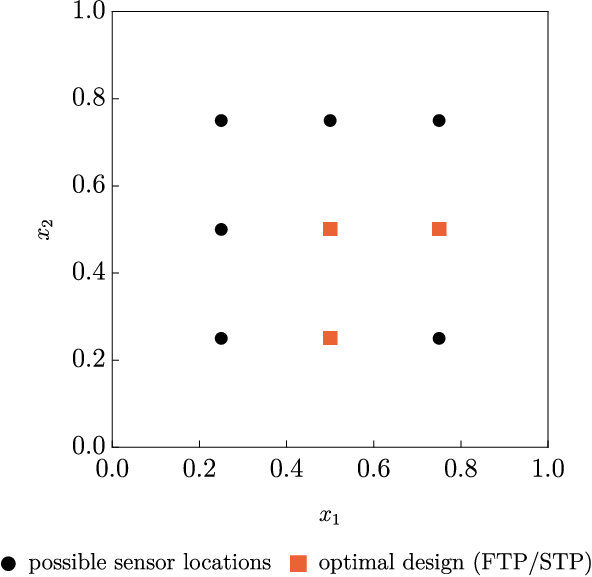}
\caption{Left: The root-mean-square (R.M.S.) cubature errors of the FTP method (a) and STP method (b) applied to the PDE problem~\eqref{eq:numexpde} subject to an affine and uniform representation of the input random field (i) corresponding to the optimal design. Right: The optimal design obtained using both approaches.}\label{fig:1}
%\end{figure}
%\vspace*{-.85cm}
%\begin{figure}[!h]
\subfloat{
\begin{tikzpicture}
\node (img) {\includegraphics[height=.365\textwidth,trim=.5cm .5cm 0cm 0cm,clip]{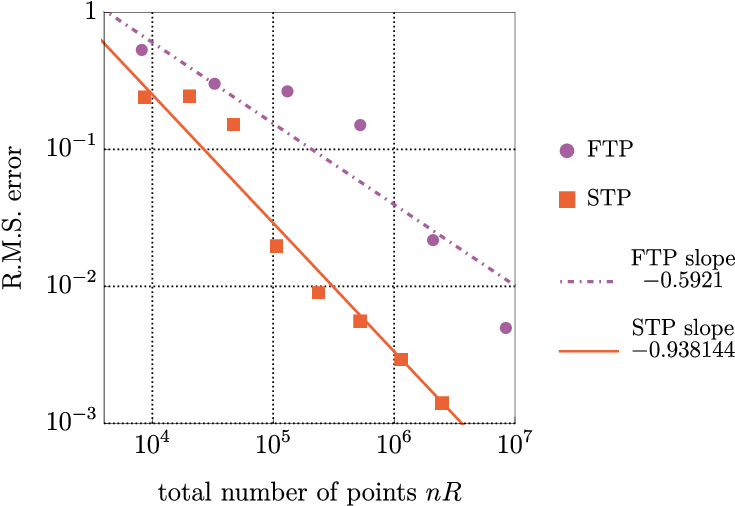}};
\node[below=of img,node distance=0cm,xshift=-.55cm,yshift=1.2cm]{\!\!\!\!\!total number of points $NR$};
\node[left=of img,node distance=0cm,rotate=90,anchor=center,xshift=.25cm,yshift=-1.0cm]{R.M.S. error};
\end{tikzpicture}}
\includegraphics[height=.428\textwidth,trim=0cm 0cm 0cm -0.85cm,clip]{periodic_design.eps}
\caption{Left: The root-mean-square (R.M.S.) cubature errors of the FTP method (a) and STP method (b) applied to the PDE problem~\eqref{eq:numexpde} subject to a periodic representation of the input random field (ii) corresponding to the optimal design. Right: The optimal design obtained using both approaches.}\label{fig:2}
\end{figure}

\begin{remark}\rm
The conditions of Assumptions~\ref{assumptionone}--\ref{assumptiontwo} are satisfied due to the truncation of the data domain~\ref{eq:A21}, the regularity analysis presented in Lemmata~\ref{lemma:innerreg} and~\ref{reg:outer}~\ref{eq:A22} and the mixed regularity analysis presented in Section~\ref{sec:mixed}~\ref{eq:A23}.% 
\end{remark}

In all experiments, we compute the value of the EIG for each $\bsxi\in \Xi$ and as the optimal design, we choose the design $\bsxi^*\in\Xi$ minimizing the value of the objective function corresponding to the largest number of cubature points for each experiment (a)--(c). 

As the generating vector for both integrals in cases (a) and (b), as well as the inner integral in part (c), we used the off-the-shelf lattice rule~\cite[\tt lattice-32001-1024-1048576.3600]{kuogeneratingvector}. For each cubature node, the PDE was solved using a first-order finite element method with mesh width $h=2^{-5}$. The root-mean-square error was approximated with respect to $R=16$ random shifts for experiments (a) and (b), and the results these experiments subject to input random field (i) are given in Figure~\ref{fig:1}, while the corresponding results for the input random field (ii) are given in Figure~\ref{fig:2}.

The convergence rate subject to the full tensor product cubature scheme is close to $\mathcal O(N^{-1/2})$ while the convergence rates for the sparse tensor product cubature scheme are nearly $\mathcal O(N^{-1})$. The results computed using the periodic parameterization appear to have a slightly improved rate of decay compared to the affine and uniform parameterization.

The results for experiment (c) are given in Figure~\ref{fig:3}. We approximated the inner integral using a lattice rule with a single random shift and, instead of estimating the root-mean-square error, we obtained the absolute errors of the FTP and STP methods by computing the difference against reference solutions corresponding to $17\,989\,120$ nodes (FTP) and $468\,732$ nodes (STP).

Since both the inner and outer integral are now approximated by higher-order cubatures, the convergence rate subject to the full tensor product cubature scheme is close to $\mathcal O(N^{-1})$ while the sparse tensor product construction achieves a convergence order of roughly $\mathcal O(N^{-2})$. We note that the preasymptotic regimes are relatively long, so the linear fits were constructed using the last three data points for the FTP method and the last five data points for the STP method.

\begin{figure}[!t]
\subfloat{
\begin{tikzpicture}
\node (img) {\includegraphics[height=.36\textwidth,trim=.5cm .5cm 0cm 0cm,clip]{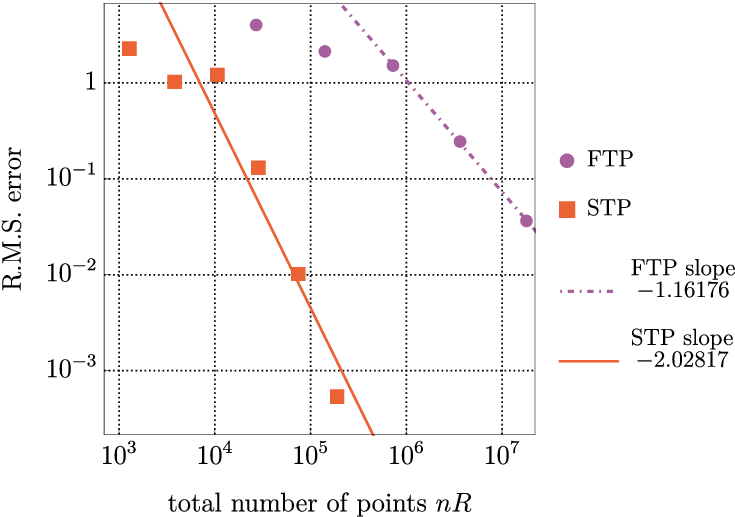}};
\node[below=of img,node distance=0cm,xshift=-.55cm,yshift=1.2cm]{\,total number of points $N$};
\node[left=of img,node distance=0cm,rotate=90,anchor=center,xshift=.25cm,yshift=-1.0cm]{absolute error};
\end{tikzpicture}}
\includegraphics[height=.43\textwidth,trim=0cm 0cm 0cm -0.85cm,clip]{periodic_design.eps}
\caption{Left: The absolute errors of the FTP and STP methods using approach (c) applied to the PDE problem~\eqref{eq:numexpde} subject to a periodic representation of the input random field (ii) corresponding to the optimal design. Right: The optimal design obtained using both approaches.}\label{fig:3}
\end{figure}

\begin{remark}\rm
Alternatively, one could use any higher-order cubature method such as interlaced polynomial lattice rules (cf., e.g.,~\cite{spodpaper14}) to approximate the inner or outer integrals. The regularity analysis developed in Section~\ref{sec:reg} can be adapted to construct tailored lattice rules for this class of quasi-Monte Carlo methods as well.\end{remark}

\section{Conclusions}\label{sec:conclusions}
In summary, this paper represents a significant advancement in the field of BOED for problems governed by PDEs. By establishing parametric regularity, we have delved deeper into the nuances of the design problem, enriching our comprehension of its underlying dynamics. Moreover, our thorough error analysis of the full tensor QMC method has showcased its robustness and efficacy, with convergence rates remaining independent of parameter dimensions.

The introduction of the sparse tensor method has unveiled considerable potential, providing a promising avenue for enhancing convergence rates in nested integrals and recovering original rates. Through numerical verification of predicted convergence rates for a specific elliptic problem, we have furnished empirical validation to support our theoretical findings, affirming the practical feasibility of our proposed methodologies.

The analysis of the sparse tensor approach for nonlinear functions within the inner integral is particularly intriguing, offering avenues for exploration in other domains such as machine learning and statistics, where nested expectations are prevalent, such as variational autoencoders or probabilistic programming systems. Future research endeavors will focus on extending our findings to encompass a broader spectrum of forward problems, not limited to the elliptic model problem.

While our analysis has demonstrated the independence of convergence behavior on parameter dimensions under suitable assumptions, the dependence on data dimensions and noise covariance could be pivotal, especially in scenarios involving informative or sequential data collection processes. In forthcoming studies, we aim to explore techniques grounded in preconditioners to mitigate this effect, building upon prior works in the field \cite{DBLP:journals/nm/SchillingsSW20}.
\section*{Acknowledgement}
CS acknowledges support from MATH+ project EF1-19: Machine Learning Enhanced Filtering Methods for Inverse Problems and EF1-20: Uncertainty Quantification and Design of Experiment for Data-Driven Control, funded by the Deutsche Forschungsgemeinschaft (DFG, German Research
Foundation) under Germany's Excellence Strategy -- The Berlin Mathematics
Research Center MATH+ (EXC-2046/1, project ID: 390685689). \rev{The work of VK was supported by the Research Council of Finland (Flagship of Advanced Mathematics for Sensing, Imaging and Modelling grant 359183).}
\appendix
\section{Technical results}\label{sec:hypergeometric}
\normalsize
The two summation identities appearing in the regularity analysis of Subsection~\ref{sec:inner} can be established using hypergeometric summation.

\begin{lemma}\label{lemma:celine} Let $v\geq 1$. Then
$$
\sum_{\ell=1}^v \frac{1}{(v-\ell)!(\ell-1)!}=\frac{2^{v-1}}{(v-1)!}.
$$
\end{lemma}
\proof  We prove this using Sister Celine's method~\cite{aeqb}. We define $T(v):=\sum_{\ell=1}^v \frac{1}{(v-\ell)!(\ell-1)!}$ and $F(v,\ell):=\frac{1}{(v-\ell)!(\ell-1)!}$. Letting $a,b,c\in\mathbb R$ be undetermined coefficients, we first seek a non-trivial solution to
\begin{align}
aF(v,\ell)+bF(v,\ell+1)+cF(v+1,\ell+1)=0.\label{eq:celine1}
\end{align}
Plugging in the values of $F$ into the above formula and regrouping the equation as a polynomial in terms of $\ell$ yields
$$
(a-b)\ell+bv+c=0.
$$
This yields $b=a$, and $c=-av$ for $a\in\mathbb R$. The relation~\eqref{eq:celine1} thus simplifies to
$$
F(v,\ell)=vF(v+1,\ell+1)-F(v,\ell+1).
$$
By taking the sum over $\ell\in\{1,\ldots,v\}$, we obtain
$$
T(v)=vT(v+1)-T(v)-vF(v+1,1)+F(v,1)-F(v,v+1).
$$
Noting that $F(v+1,1)=\frac{1}{v!}$, $F(v,1)=\frac{1}{(v-1)!}$, and $F(v,v+1)=0$ (by convention\footnote[1]{It is not difficult to check that this convention satisfies the contiguous relation $F(v,\ell)=vF(v+1,\ell+1)-F(v,\ell+1)$ with $\ell=v$.}), we obtain the recurrence
$$
T(v+1)=\frac{2}{v}T(v),\quad T(1)=1.
$$
The claim is an immediate consequence of this recurrence relation.\qed\endproof

\begin{lemma}\label{lemma:gosper} Let $v\geq 1$ and $2\leq \lambda\leq v+1$. Then
\begin{align}
\sum_{\ell=\lambda-1}^{v}(v-\ell+1)\frac{(\ell-1)!}{(\ell-\lambda+1)!}=\frac{(v+1)!}{(v-\lambda+1)!\lambda(\lambda-1)}.\label{eq:gosper0}
\end{align}
\end{lemma}
\proof We start by proving a different yet related identity:
\begin{align}
\sum_{\ell=b}^v (a-\ell) \frac{(\ell-1)!}{(\ell-b)!}=\frac{(b-v-1)(bv+b-ab-a)v!}{b(1+b)(v-b+1)!}.\label{eq:gosper1}
\end{align}
Let $t_{\ell}:=(a-\ell)\frac{(\ell-1)!}{(\ell-b)!}$. It turns out that the left-hand side expression in~\eqref{eq:gosper1} is {\em Gosper-summable}: it is straightforward to check that $t_{\ell}=z_{\ell+1}-z_{\ell}$, where
$$
z_{\ell}:=\frac{(a+ab-b\ell)(\ell-1)!}{b(1+b)(\ell-b-1)!}.
$$
In consequence,
\begin{align*}
\sum_{\ell=b}^v (a-\ell) \frac{(\ell-1)!}{(\ell-b)!} = \sum_{\ell=b}^v t_\ell=z_{v+1}-z_b=\frac{(b-v-1)(bv+b-ab-a)v!}{b(1+b)(v-b+1)!},
\end{align*}
as desired.

The equation~\eqref{eq:gosper0} follows by substituting $a=v+1$ and $b=\lambda-1$ into~\eqref{eq:gosper1}.\qed\endproof

\section{{QMC analysis under periodic change of variables}}\label{sec:periodic}

We begin by proving a general result on the parametric regularity bounds for smooth Banach space valued functions under a periodic change of variables.
\begin{theorem}\label{thm:periodicregularity}
Let $X$ be a separable Banach space. Let $(\Gamma_k)_{k\geq 0}$ and $\boldsymbol b=(b_j)_{j\geq 1}$ be sequences of nonnegative numbers and $C_0>0$. Suppose that $u(\boldsymbol \theta)\in X$ is infinitely many times continuously differentiable such that
$$
\|\partial_{\bstheta}^{\bsnu}u(\bstheta)\|_X\leq C_0\Gamma_{|\bsnu|}\bsb^{\bsnu}\quad\text{for all}~\bsnu\in\mathbb N_0^s~\text{and}~\bstheta\in \Theta_s:=[-1/2,1/2]^s.
$$
Then the function defined by
\begin{align}\label{eq:comp}
u_{\rm per}(\bstheta):=u(\sin(2\pi \bstheta)),\quad \bstheta\in \Theta_s,
\end{align}
satisfies the regularity bound
$$
\|\partial_{\bstheta}^{\boldsymbol\nu} u_{\rm per}(\boldsymbol \theta)\|_X\leq (2\pi)^{|\bsnu|}C_0\sum_{\boldsymbol m\leq\boldsymbol\nu}\Gamma_{|\boldsymbol m|}\boldsymbol b^{\boldsymbol m}\prod_{j\geq 1}S(\nu_j,m_j)
$$
for all $\boldsymbol\nu\in\mathbb N_0^s$ and $\boldsymbol \theta\in \Theta_s$, where $S(\cdot,\cdot)$ denotes the \emph{Stirling number of the second kind}.
\end{theorem}
We begin by outlining the proof strategy. The composition~\eqref{eq:comp} suggests using Fa\`a di Bruno's formula~\cite{savits}: for $\boldsymbol\nu\in\mathbb N_0^s\setminus\{\mathbf 0\}$, there holds
\begin{align}
\partial_{\bstheta}^{\boldsymbol\nu}u_{\rm per}(\boldsymbol \theta)=\sum_{\substack{\boldsymbol\lambda\in\mathbb N_0^s\\ 1\leq|\boldsymbol\lambda|\leq |\boldsymbol\nu|}}\partial_{\boldsymbol{\theta'}}^{\boldsymbol\lambda}u(\boldsymbol{\theta'})\bigg|_{\boldsymbol{\theta'}=\sin(2\pi \boldsymbol\theta)}\kappa_{\boldsymbol\nu,\boldsymbol\lambda}(\boldsymbol \theta),\label{eq:faadibruno2}
\end{align}
where the sequence $(\kappa_{\boldsymbol\nu,\boldsymbol\lambda}(\boldsymbol \theta))$ is defined recursively by
\begin{align*}
&\kappa_{\boldsymbol\nu,\mathbf 0}\equiv \delta_{\boldsymbol\nu,\mathbf 0},\\
&\kappa_{\boldsymbol\nu,\boldsymbol\lambda}\equiv 0\quad\text{if}~|\boldsymbol\nu|<|\boldsymbol\lambda|~\text{or}~\boldsymbol\lambda\not\geq\mathbf 0,\\
&\kappa_{\boldsymbol\nu+\boldsymbol e_j,\boldsymbol\lambda}(\boldsymbol \theta)=\sum_{\ell\in{\rm supp}(\boldsymbol\lambda)}\sum_{\boldsymbol m\leq\boldsymbol\nu}\binom{\boldsymbol\nu}{\boldsymbol m}\partial_{\bstheta}^{\boldsymbol m+\boldsymbol e_j}\sin(2\pi \theta_\ell)\kappa_{\boldsymbol\nu-\boldsymbol m,\boldsymbol\lambda-\boldsymbol e_{\ell}}(\boldsymbol \theta)\\
&=\sum_{m_j=0}^{\nu_j}\binom{\nu_j}{m_j}\partial_{\theta_j}^{m_j+1}\sin(2\pi \theta_j)\kappa_{\boldsymbol\nu-m_j\boldsymbol e_j,\boldsymbol\lambda-\boldsymbol e_j}(\boldsymbol \theta)\quad\text{otherwise}.
\end{align*}
Making use of the fact that
$$
\bigg|\frac{{\rm d}^k}{{\rm d}\theta^k}\sin(2\pi \theta)\bigg|=\bigg|(2\pi)^k\sin\bigg(2\pi \theta+k\,\frac{\pi}{2}\bigg)\bigg|\leq (2\pi)^k,
$$
we can find an upper bound for the sequence $(\kappa_{\boldsymbol\nu,\boldsymbol\lambda}(\boldsymbol \theta))$ defined by an auxiliary sequence $(\rev{\chi}_{\boldsymbol\nu,\boldsymbol\lambda})$ given by the recursion
\begin{align*}
&\rev{\chi}_{\boldsymbol\nu,\mathbf 0}=\delta_{\boldsymbol\nu,\mathbf 0},\\
&\rev{\chi}_{\boldsymbol\nu,\boldsymbol\lambda}=0\quad\text{if}~|\boldsymbol\nu|<|\boldsymbol\lambda|~\text{or}~\boldsymbol\lambda\not\geq\mathbf 0,\\
&\rev{\chi}_{\boldsymbol\nu+\boldsymbol e_j,\boldsymbol\lambda}=\sum_{m_j=0}^{\nu_j}\binom{\nu_j}{m_j}(2\pi)^{m_j+1}\rev{\chi}_{\boldsymbol\nu-m_j\boldsymbol e_j,\boldsymbol\lambda-\boldsymbol e_j}\quad\text{otherwise.}
\end{align*}
This sequence has the following closed form solution.
\begin{lemma}\label{lemma:stirling} There holds
$$
\rev{\chi}_{\boldsymbol\nu,\boldsymbol\lambda}=(2\pi)^{|\boldsymbol\nu|}\prod_{j\geq 1}S(\nu_j,\lambda_j)\quad\text{for all}~\boldsymbol\nu,\boldsymbol\lambda\in\mathbb N_0^s.
$$
\end{lemma}
\proof Let $\boldsymbol\lambda\in\mathbb N_0^s$ be arbitrary. The proof is carried out by induction with respect to the modulus of $\boldsymbol\nu\in\mathbb N_0^s$. The base step is resolved by observing that
$$
\rev{\chi}_{\mathbf 0,\mathbf 0}=1=\prod_{j\geq 1}S(0,0),%
$$
and, if $\boldsymbol\lambda\neq\mathbf 0$,
$$
\rev{\chi}_{\mathbf 0,\boldsymbol\lambda}=0=\prod_{j\geq 1}S(0,\lambda_j),
$$
where the second inequality holds due to ${\rm supp}(\boldsymbol\lambda)\neq\varnothing$.

To resolve the induction step, let $\boldsymbol\nu\in\mathbb N_0^s$ and suppose that the claim has already been proved for all multi-indices with modulus less than or equal to $|\boldsymbol\nu|$. Let $j\geq 1$ be arbitrary. Then
\begin{align*}
\rev{\chi}_{\boldsymbol\nu+\boldsymbol e_j,\boldsymbol\lambda}&=\sum_{m_j=0}^{\nu_j}\binom{\nu_j}{m_j}(2\pi)^{m_j+1}(2\pi)^{|\boldsymbol\nu|-m_j}S(\nu_j-m_j,\lambda_j-1)\prod_{i\neq j}S(\nu_i,\lambda_i)\\
&=(2\pi)^{|\boldsymbol\nu|+1}\prod_{i\neq j}S(\nu_i,\lambda_i)\sum_{m_j=0}^{\nu_j}\binom{\nu_j}{m_j}S(\nu_j-m_j,\lambda_j-1)\\
&=(2\pi)^{|\boldsymbol\nu|+1}\bigg(\prod_{i\neq j}S(\nu_i,\lambda_i)\bigg)S(\nu_j+1,\lambda_j),
\end{align*}
where the final equality is an immediate consequence of~\cite[formula 26.8.23]{nist}.\qed\endproof

\proof[Proof of Theorem~\ref{thm:periodicregularity}] Since $\kappa_{\boldsymbol\nu,\boldsymbol\lambda}\leq \rev{\chi}_{\boldsymbol\nu,\boldsymbol\lambda}$ holds by construction, we may plug into~\eqref{eq:faadibruno2} the identity proved in Lemma~\ref{lemma:stirling} and use the bound~(A2) to obtain
\begin{align*}
\|\partial_{\bstheta}^{\boldsymbol\nu} u_{\rm per}(\boldsymbol \theta)\|_X&\leq (2\pi)^{|\boldsymbol\nu|}C_0\sum_{\substack{\boldsymbol\lambda\in\mathbb N_0^s\\ 1\leq|\boldsymbol\lambda|\leq |\boldsymbol\nu|}}\Gamma_{|\boldsymbol\lambda|}\boldsymbol b^{\boldsymbol\lambda}\prod_{j\geq 1}S(\nu_j,\lambda_j)\\
&=(2\pi)^{|\boldsymbol\nu|}C_0\sum_{\boldsymbol\lambda\leq\boldsymbol\nu}\Gamma_{|\boldsymbol\lambda|}\boldsymbol b^{\boldsymbol\lambda}\prod_{j\geq 1}S(\nu_j,\lambda_j)
\end{align*}
since $\prod_{j\geq 1}S(\nu_j,\lambda_j)=0$ whenever $\boldsymbol\lambda\not\leq\boldsymbol\nu$.\qed\endproof

\begin{remark}\rm
The Fa\`a di Bruno formula in was developed for scalar-valued functions in~\cite{savits}, but we applied it above for Banach space valued functions. This is not an issue as can be seen by the following simple argument: for arbitrary $G\in X'$, there holds
\begin{align*}
\langle G,\partial_{\bstheta}^{\bsnu} u_{\rm per}(\bstheta)\rangle_{X',X}&=\partial_{\bstheta}^{\bsnu}\langle G,u_{\rm per}(\bstheta)\rangle_{X',X}\\
&=\partial_{\bstheta}^{\bsnu}(\langle G,u(\cdot)\rangle_{X',X}\circ \sin(2\pi \bstheta))\\
&=\sum_{\substack{\boldsymbol\lambda\in\mathbb N_0^s\\ 1\leq|\boldsymbol\lambda|\leq |\bsnu|}}\partial_{\boldsymbol{\theta'}}^{\boldsymbol\lambda}\langle G,u(\boldsymbol{\theta'})\rangle_{X',X}\bigg|_{\boldsymbol{\boldsymbol\theta'}=\sin(2\pi \bstheta)}\kappa_{\bsnu,\boldsymbol\lambda}(\bstheta)\\
&=\bigg\langle G,\sum_{\substack{\boldsymbol\lambda\in\mathbb N_0^s\\ 1\leq|\boldsymbol\lambda|\leq |\bsnu|}}\partial_{\boldsymbol{\theta'}}^{\boldsymbol\lambda}u(\boldsymbol{\theta'})\bigg|_{\boldsymbol{\theta'}=\sin(2\pi \bstheta)}\kappa_{\bsnu,\boldsymbol\lambda}(\bstheta)\bigg\rangle_{X',X},
\end{align*}
where the scalars $(\kappa_{\bsnu,\boldsymbol\lambda}(\bstheta))$ are defined using {\em exactly} the same recursion as before. Since the above derivation holds for all $G\in X'$, we conclude that Fa\`a di Bruno's formula~\eqref{eq:faadibruno2} is valid for Banach space valued functions.\end{remark}

The significance of the preceding result can be understood as follows: in order to obtain the parametric regularity bound for a given problem under the periodic paradigm, it is in principle sufficient to carry out the parametric regularity analysis under the assumption of an underlying affine and uniform random field and then apply Theorem~\ref{thm:periodicregularity} to obtain the corresponding regularity bound for the periodically transformed problem.

As a corollary, we obtain the following analogues of Lemmata~\ref{lemma:innerreg} and~\ref{lemma:mixedregbound} for the periodic model problem~\eqref{eq:periodicmodelproblem}.
\begin{lemma}Let $\bsnu\in\mathbb N_0^s\setminus\{\mathbf0\}$. Then under assumptions~{\rm \ref{eq:A2}}--{\rm \ref{eq:A3}}, there holds for~\eqref{eq:periodicmodelproblem} that
\begin{align*}
&|\partial_{\boldsymbol\theta}^{\boldsymbol\nu}{\rm e}^{-\frac{1}{2}\|\boldsymbol y-G_{s,{\rm per}}(\boldsymbol\theta,\boldsymbol\xi)\|_{\Gamma^{-1}}^2}|\\
&\leq 3.82^k(2\pi)^{|\bsnu|}\sum_{\boldsymbol m\leq \bsnu} C^{|\boldsymbol{m}|}\rev{2^{\beta(|\boldsymbol{m}|-1)}}\mu_{\min}^{-|\boldsymbol{m}|/2}\rev{(|\boldsymbol{m}|!)^{\beta}}\boldsymbol b^{\boldsymbol{m}}\prod_{j\geq 1}S(\nu_j,m_j)
\end{align*}
for all $\bstheta\in\Theta_s$, $\bsy\in [-K,K]^k$, and $\bsxi\in\Xi$.
\end{lemma}

\begin{lemma}Let $\bsnu\in\mathbb N_0^s\setminus\{\mathbf 0\}$ and $\bseta\in\mathbb N_0^k\setminus\{\mathbf 0\}$. Then under assumptions~{\rm \ref{eq:A2}}--{\rm \ref{eq:A3}}, there holds for~\eqref{eq:periodicmodelproblem} that
\begin{align*}
&|\partial_{\bsy}^{\bseta}\partial_{\bstheta}^{\bsnu}{\rm e}^{-\frac12 \|\bsy-G_{s,{\rm per}}(\bstheta,\bsxi)\|_{\Gamma^{-1}}^2}|\\
&\!\leq\! \rev{1.1^k 2^{\beta(k-1)}}k{\rm e}^{2k}\bseta!(2\pi)^{|\bsnu|}\!\sum_{\boldsymbol m\leq\bsnu}\!(\rev{4^\beta}C)^{|\boldsymbol m|}\mu_{\min}^{-|\boldsymbol m|-1/2}R^{|\boldsymbol m|}\rev{(|\boldsymbol m|!)^\beta}\bsb^{\boldsymbol m}\!\prod_{j\geq 1}\!S(\nu_j,m_j)
\end{align*}
for all $\bstheta\in\Theta_s$, $\bsy\in[-K,K]^k$, and $\bsxi\in\Xi$.
\end{lemma}

Let $F\in C([0,1)^s)$ be a smooth, 1-periodic function with dominating mixed smoothness of order $\alpha>1$ and consider the cubature rule
$$
Q_{s,n}(F):=\frac{1}{n}\sum_{i=1}^n F(\bst_i)\approx \int_{[0,1]^s}F(\bstheta)\,{\rm d}\bstheta=:I_s(F)
$$
over (unshifted) lattice points~\eqref{eq:latticepoint}. By defining the norm
$$
\|F\|_{\mathcal K_{s,\alpha,\bsgamma}}:=\sup_{\boldsymbol h\in\mathbb Z^s}|\widehat F(\boldsymbol h)|r_{\alpha}(\boldsymbol\gamma,\boldsymbol h),\quad r_{\alpha}(\bsgamma,\boldsymbol h):=\gamma_{{\rm supp}(\boldsymbol h)}^{-1}\prod_{j\in{\rm supp}(\boldsymbol h)}|h_j|^\alpha,
$$
where ${\rm supp}(\boldsymbol h):=\{j\in\{1:s\}:h_j\neq 0\}$ and $\widehat F(\boldsymbol h):=\int_{[0,1]^s}F(\bstheta){\rm e}^{-2\pi{\rm i}\bstheta\cdot \boldsymbol h}\,{\rm d}\bstheta$ for $\boldsymbol h\in\mathbb Z^s$, and $\bsgamma=(\gamma_{\setu})_{\setu\subseteq\{1:s\}}$ denotes a collection of positive weights, we have the following.

\begin{lemma}[cf.~{\cite{korobovpaper}}]
Let $s\in\mathbb N$, $\alpha>1$, and let $\bsgamma=(\gamma_{\setu})_{\setu\subseteq\{1:s\}}$ be a collection of positive weights. Let $F\in C([0,1)^s)$ be a 1-periodic function with respect to each of its variables such that $\|F\|_{\mathcal K_{s,\alpha,\bsgamma}}<\infty$. An $s$-dimensional lattice rule with $n=2^m$ points, $m\geq 0$, can be constructed by a CBC algorithm such that, for all $\lambda\in (1/\alpha,1]$,
$$
|I_{s}(F)-Q_{s,n}(F)|\leq\bigg(\frac{2}{n}\sum_{\varnothing\neq \mathfrak u\subseteq\{1:s\}}\gamma_{\mathfrak u}^\lambda(2\zeta(\alpha\lambda))^{|\setu|}\bigg)^{1/\lambda}\|F\|_{\mathcal K_{s,\alpha,\boldsymbol\gamma}},
$$
where $\zeta(x):=\sum_{\ell=1}^\infty \ell^{-x}$ is the {Riemann zeta function} for $x>1$.
\end{lemma}

When $\alpha\geq 2$ is an integer, there holds
$$
\|F\|_{\mathcal K_{s,\alpha,\bsgamma}}\leq \max_{\setu\subseteq\{1:s\}}\frac{1}{(2\pi)^{\alpha|\setu|}}\frac{1}{\gamma_{\setu}}\int_{[0,1]^{|\setu|}}\bigg|\int_{[0,1]^{s-|\setu|}}\bigg(\prod_{j\in\setu}\frac{\partial^{\alpha}}{\partial y_j^{\alpha}}\bigg)F(\bsy)\,{\rm d}\bsy_{-\setu}\bigg|\,{\rm d}\bsy_{\setu}
$$
provided that $F$ has mixed partial derivatives of order $\alpha$.

We consider the parametric regularity of the inner integrand appearing in the expression
\begin{align}\label{eq:innerintdefper}
\int_{\Theta_s}f_{\rm per}(\bstheta,\bsy)\,{\rm d}\bstheta,\quad f_{\rm per}(\bstheta,\bsy):=C_{k,\Gamma}{\rm e}^{-\frac12 \|\bsy-G_{s,{\rm per}}(\bstheta,\bsxi)\|_{\Gamma^{-1}}^2}.
\end{align}
In complete analogy to the derivation in~\cite{KKS}, we obtain the following result.
\begin{theorem}
Let $n=2^m$, $m\geq 0$. Then under assumptions {\rm \ref{eq:A2}}--{\rm \ref{eq:A3}}, it is possible to use a CBC algorithm to obtain a generating vector $\boldsymbol z\in\mathbb N^s$ such that the rank-1 lattice rule for the integrand $f_{\rm per}$ of~\eqref{eq:innerintdefper} satisfies the root-mean-square error estimate
$$
|I_s(f_{\rm per})-{Q}_{s,n}(f_{\rm per})|\leq Cn^{-1/p},
$$
where the constant $C>0$ is independent of the dimension $s$, provided that the smoothness-driven product and order dependent (SPOD) weights
$$
\gamma_{\setu}:=\sum_{\boldsymbol m_{\setu}\in\{1:\alpha\}^{|\setu|}}C^{|\boldsymbol m_{\setu}|}\rev{2^{\beta(|\boldsymbol m_{\setu}|-1)}}\mu_{\min}^{-|\boldsymbol m_{\setu}|/2}\rev{(|\boldsymbol m_{\setu}|!)^\beta}\prod_{j\in \setu}\big(b_j^{m_j}S(\alpha,m_j)\big)
$$
are used as inputs to the CBC algorithm with $\alpha=\lfloor 1/p\rfloor+1$\rev{ and $p<\frac1{\beta}$}.
\end{theorem}

The convergence rates for the full tensor product and sparse tensor product approximations of the double integral subject to the periodically parameterized forward model, i.e.,
$$
\int_{[-K,K]^k} g\bigg(\int_{\Theta_s} f_{\rm per}(\bstheta,\bsy)\,{\rm d}\bstheta\bigg)\,{\rm d}\bsy,\quad g(y)=x\log x,
$$coincide with those presented in Theorems~\ref{lemma:ftprate} and~\ref{lemma:genericstp} when the outer integral is discretized using a first-order method. However, if the outer integral is approximated using a higher-order cubature method---such that its rate is balanced with the higher-order rate exhibited by the periodically parameterized inner integral---then the statements of Theorems~\ref{lemma:ftprate} and~\ref{lemma:genericstp} hold true with the obvious substitution of higher-order convergence rates in place of the first-order rates. In particular, the dimension independence can be established. Furthermore, the sparse tensor product can recover the optimal rate up to a logarithmic factor. We demonstrate these effects in the numerical experiments of Section~\ref{sec:numex}.

\bibliographystyle{spmpsci}
\bibliography{doe}

\end{document}